\documentclass[11pt]{amsart}

\usepackage[pdftex]{graphicx}
\usepackage[dvipsnames]{xcolor}

\usepackage[margin=3cm]{geometry}

\usepackage{amsfonts}
\usepackage{amsmath}
\usepackage{amssymb}
\usepackage{amsthm}

\usepackage{paralist}
\usepackage[colorlinks,cite color=blue,pagebackref=true,pdftex]{hyperref}

\usepackage[T1]{fontenc}

\usepackage{caption}
\usepackage{subcaption}

\newcommand\Defn[1]{\textbf{\color{black}#1}}

\renewcommand\u{\mathbf{u}}
\renewcommand\v{\mathbf{v}}

\newcommand\p{\mathbf{p}}

\newcommand\FaPa{\mathcal{B}}
\newcommand\Ptop{\widehat{1}}
\newcommand\Pbot{\widehat{0}}
\newcommand\NC{N}

\renewcommand\emptyset{\varnothing}
\newcommand\Z{\mathbb{Z}}               
\newcommand\Znn{\mathbb{Z}_{\ge0}}               

\newcommand\R{\mathbb{R}}               
\newcommand\Rnn{\R_{\ge0}}               
\newcommand\C{\mathbb{C}}               

\newcommand\z{\mathbf{z}}
\newcommand\x{\mathbf{x}}
\newcommand\y{\mathbf{y}}

\newcommand\subdiv{\mathcal{S}}
\newcommand\init{\textrm{in}_{\le_{rev}}}

\newcommand\CaySum[2]{{#1 \boxplus #2}}%
\newcommand\CayDiff[2]{{#1 \boxminus #2}}%
\newcommand\tprism[1]{\mathrm{tprism}(#1)}%

\newcommand\op{\mathrm{op}}
\renewcommand\P{P}
\newcommand\dP{\mathbf{P}}
\newcommand\Phat{\widehat{\P}}

\newcommand\AffLat{\mathbb{A}}

\newcommand\Ord[1]{\mathcal{O}(#1)}
\newcommand\Chain[1]{\mathcal{C}({#1})}

\newcommand\TOrd[1]{\Ord{#1}}
\newcommand\DOrd[1]{\overline{\mathcal{O}}(#1)}

\newcommand\TChain[1]{\Chain{#1}}
\newcommand\DChain[1]{\overline{\mathcal{C}}(#1)}

\newcommand\Transfer{\phi}
\newcommand\iTransfer{\psi}
\newcommand\TPsi[1]{\Psi_{#1}}%

\newcommand\ehr{\mathrm{Ehr}}%

\newcommand\Hring[1]{\C[\Ord{#1}]}
\newcommand\dHring[1]{\C[\TOrd{#1}]}
\newcommand\Id{\mathrm{I}}
\newcommand\mo{\le_{rev}}

\newcommand\dG{\mathbf{G}}
\newcommand\Hansen[1]{\mathcal{H}(#1)}

\newcommand\Val{\mathrm{Val}}
\newcommand\ValP{\Val_0}

\newcommand\Po{\mathcal{P}}
\newcommand\Qo{\mathcal{Q}}

\newcommand\0{\mathbf{0}}
\newcommand\1{\mathbf{1}}
\newcommand\e{\mathbf{e}}
\newcommand\dual{\triangle}

\newcommand\Birkhoff{\mathcal{J}}
\newcommand\Birk[1]{\Birkhoff(#1)}
\newcommand\TBirk[1]{\Birk{#1}}
\renewcommand\L{\mathrm{L}}
\newcommand\inner[1]{\langle{#1}\rangle}
\newcommand\Filter{\mathsf{J}}

\DeclareMathOperator{\supp}{supp}
\DeclareMathOperator{\sgn}{sgn}
\DeclareMathOperator{\relint}{relint}

\DeclareMathOperator{\conv}{conv}
\DeclareMathOperator{\cone}{cone}
\DeclareMathOperator{\nvol}{Vol}
\DeclareMathOperator{\vol}{vol}

\newtheorem{thm}{Theorem}[section]
\newtheorem{cor}[thm]{Corollary}
\newtheorem{lem}[thm]{Lemma}
\newtheorem{prop}[thm]{Proposition}
\newtheorem{conj}{Conjecture}

\theoremstyle{definition}

\newtheorem{example}{Example}
\newtheorem{rem}{Remark}

\title{Two double poset polytopes}

\author{Thomas Chappell}
\email{tom@jjdat.com}

\author{Tobias Friedl}
\address{Fachbereich Mathematik und Informatik, %
Freie Universit\"at Berlin, %
Germany}
\email{tfriedl@math.fu-berlin.de}

\author{Raman Sanyal}
\address{Institut f\"ur Mathematik, Goethe-Universit\"at Frankfurt, Germany}
\email{sanyal@math.uni-frankfurt.de}


\date{\today}

\begin{document}

\begin{abstract}
    To every poset $\P$, Stanley (1986) associated two polytopes, the order
    polytope and the chain polytope, whose geometric properties reflect the
    combinatorial qualities of $\P$. This construction allows for deep
    insights into combinatorics by way of geometry and vice versa.  Malvenuto
    and Reutenauer (2011) introduced \emph{double posets}, that is, (finite)
    sets equipped with two partial orders, as a generalization of Stanley's
    labelled posets. Many combinatorial constructions can be naturally phrased
    in terms of double posets.  We introduce the \emph{double order polytope}
    and the \emph{double chain polytope} and we amply demonstrate that they
    geometrically capture double posets, i.e., the interaction between the two
    partial orders. We describe the facial structures, Ehrhart polynomials, and
    volumes of these polytopes in terms of the combinatorics of double posets.
    We also describe a curious connection to Geissinger's valuation polytopes
    and we characterize $2$-level polytopes among our double poset polytopes.

    Fulkerson's \emph{anti-blocking} polytopes from combinatorial optimization
    subsume stable set polytopes of graphs and chain polytopes of posets. We
    determine the geometry of Minkowski- and Cayley sums of anti-blocking
    polytopes. In particular, we describe a canonical subdivision of Minkowski
    sums of anti-blocking polytopes that facilitates the computation of
    Ehrhart (quasi-)polynomials and volumes. This also yields canonical
    triangulations of double poset polytopes.
    
    Finally, we investigate the affine semigroup rings associated to 
    double poset polytopes. We show that they have quadratic Gr\"obner bases,
    which gives an algebraic description of the unimodular flag triangulations
    described in the first part.
\end{abstract}

\maketitle

\section{Introduction}\label{sec:intro}

A (finite) \Defn{partially ordered set} (or \Defn{poset}, for short) is a
finite set $\P$ together with a reflexive, transitive, and anti-symmetric
relation $\preceq$. The notion of \emph{partial order} pervades all of
mathematics and the enumerative and algebraic combinatorics of posets is
underlying in computations in virtually all areas. In 1986,
Stanley~\cite{TwoPoset} defined two convex polytopes for every poset $\P$
that, in quite different ways, geometrically capture combinatorial properties
of $\P$.  The \Defn{order polytope} $\Ord{\P}$ is set of all order preserving
functions into the interval $[0,1]$. That is, all functions $f : \P
\rightarrow \R$ such that 
\[
    0 \ \le \  f(a) \ \le \ f(b) \ \le \  1
\]
for all $a,b \in \P$ with $a \preceq b$.  Hence, $\Ord{\P}$ parametrizes
functions on $\P$ and many properties of $\P$ are encoded in the boundary
structure of $\Ord{\P}$: faces of $\Ord{\P}$ are in correspondence with
quotients of $\P$. In particular, the vertices of $\Ord{\P}$ are
in bijection to filters of $\P$. But also
metric and arithmetic properties of $\Ord{\P}$ can be determined from $\P$.
The order polytope naturally has vertices in the lattice $\Z^\P$ and its
Ehrhart polynomial $\ehr_{\Ord{\P}}(n) = |n\Ord{\P} \cap \Z^\P|$, up to a
shift, coincides with the order polynomial $\Omega_\P(n)$; see
Section~\ref{ssec:TO_ehrhart} for details.  
A full-dimensional simplex with vertices in a lattice $\Lambda \subset \R^n$
is \Defn{unimodular} with respect to $\Lambda$ if it has minimal volume. The
\Defn{normalized volume} relative to $\Lambda$ is the Euclidean volume scaled
such that the volume of a unimodular simplex is $1$. If the lattice is
clear from the context, we denote the normalized volume by $\nvol(\Po)$.
By describing a canonical triangulation of
$\Ord{\P}$ into unimodular simplices, Stanley showed that $\nvol(\Ord{\P})$ is
exactly the number of \Defn{linear extensions} of $\P$, that is, the number
$e(\P)$ of refinements of $\preceq$ to a total order. We will review these
results in more detail in Section~\ref{ssec:TO_ehrhart}.  This bridge between
geometry and combinatorics can, for example, be used to show that computing
volume is \emph{hard} (cf.~\cite{BW}) and, conversely, geometric inequalities
can be used on partially ordered sets; see~\cite{KN,TwoPoset}. 

The \Defn{chain polytope} $\Chain{\P}$ is the
collection of functions $g : \P \rightarrow \R_{\ge 0}$ such that 
\begin{equation}\label{eqn:chain}
    g(a_1) + g(a_2) + \cdots + g(a_k) \ \le \ 1
\end{equation}
for all \Defn{chains} $a_1 \prec a_2 \prec \cdots \prec a_k$ in $\P$. In
contrast to the order polytope, $\Chain{\P}$ does not determine $\P$. In fact,
$\Chain{\P}$ is defined by the \emph{comparability graph} of $\P$ and bears
strong relations to so-called stable set polytopes of perfect graphs; see
Section~\ref{ssec:double_graph}. Surprisingly, it is shown in~\cite{TwoPoset}
that the chain polytope and the order polytope have the same Ehrhart
polynomial and hence $\nvol(\Chain{\P}) = \nvol(\Ord{\P}) = e(\P)$, which
shows that the number of linear extensions only depends on the comparability
relation.
Stanley's poset
polytopes are very natural objects that appear in a variety of contexts in
combinatorics and beyond; see~\cite{ABS, js12, Sottile, Fourier}.

Inspired by Stanley's \emph{labelled} posets, Malvenuto and
Reutenauer~\cite{MR} introduced double poset in the context of
combinatorial Hopf algebras.  A \Defn{double poset} $\dP$ is a triple
consisting of a finite ground set $\P$ and two partial order relations
$\preceq_+$ and $\preceq_-$ on $\P$. We will write $\P_+ = (\P,\preceq_+)$ and
$\P_- = (\P,\preceq_-)$ to refer to the two underlying posets. If $\preceq_-$
is a total order, then this corresponds to labelled poset in the sense of
Stanley~\cite{Stanley72}, which is the basis for the rich theory of
$\P$-partitions. The combinatorial study of general double posets gained 
momentum in recent years with a focus on algebraic aspects; see, for
example,~\cite{Foissy1,Foissy2}. The goal of this paper is to build a bridge to
geometry by introducing \emph{two double poset polytopes} that, like the
chain- and the order polytope, geometrically reflect the combinatorial
properties of double posets and, in particular, the interaction between the
two partial orders.  

\subsection{Double order polytopes}

For a double poset $\dP = (\P,\preceq_\pm)$, we define the \Defn{double order
polytope} as
\[
    \TOrd{\dP} \ = \ \TOrd{\P,\preceq_+,\preceq_-} \ := \ \conv \bigl\{ (2\Ord{\P_+}
    \times \{1\} ) \cup (-2\Ord{\P_-} \times \{-1\} ) \bigr\}.
\]
This is a $(|\P|+1)$-dimensional polytope in $\R^\P \times \R$. Its vertices
are trivially in bijection to filters of $\P_+$ and $\P_-$. This is a lattice
polytope with respect to $\Z^\P \times \Z$ but we will mostly view
$\TOrd{\dP}$ as a lattice polytope with respect to the \Defn{affine lattice} 
$\AffLat = 2\Z^\P \times (2\Z + 1)$. That is, up to a translation by
$(\0,1)$, $\TOrd{\dP}$ is the polytope
\[
    2 \cdot \conv \bigl\{ (\Ord{\P_+} \times \{1\} ) \cup (-\Ord{\P_-} \times
    \{0\} ) \bigr\},
\]
which is a lattice polytope with respect to $2\Z^\P \times 2\Z$.  In
Section~\ref{ssec:TO_facets}, we describe the facets of $\TOrd{\dP}$ in terms
of chains and cycles alternating between $\P_+$ and $\P_-$ and, for the
important case of \emph{compatible} double posets, we completely determine the
facial structure in Section~\ref{ssec:TO_faces} in terms of \emph{double
Birkhoff lattices} $\Birk{\dP} := \Birk{\P_+} \uplus \Birk{\P_-}$. 
The double order polytope automatically has $2\Ord{\P_+}$ and $-2\Ord{\P_-}$
as facets. The non-trivial combinatorial structure is captured by 
the \Defn{reduced} double order polytope
\[
    \DOrd{\dP} \ := \ \TOrd{\dP} \cap \{ (f,t) : t = 0 \} \ = \ \Ord{\P_+} -
    \Ord{\P_-},
\]
which is a lattice polytope with respect to $\Z^\P$ by 
our choice of embedding.

By placing $2\Ord{\P_+}$ and $-2\Ord{\P_-}$ at heights $+1$ and $-1$,
respectively, we made sure that $\TOrd{\dP}$ always contains the origin.
Every poset $(\P,\preceq)$ trivially induces a double poset $\dP_\circ =
(\P,\preceq,\preceq)$ and for an induced double poset, $\TOrd{\dP_\circ}$ is
centrally-symmetric and, up to a (lattice-preserving) shear, is the polytope
\[
    \TOrd{\dP_\circ} \ \cong \ 
    \conv \bigl\{ (2\Ord{\P} \times \{1\}) \cup (2\Ord{\P^\op} \times
    \{-1\}) \bigr\},
\]
where $\P^\op$ is the poset with the opposite order.  Geissinger~\cite{Geissinger} introduced a
polytope associated to valuations on distributive lattices with values in
$[0,1]$.  In Section~\ref{ssec:val}, we show a surprising connection between
Geissinger's valuation polytopes and \emph{polars} of the (reduced) double
order polytopes of $\dP_\circ$.   We will review notions from the theory of
double posets and emphasize their geometric counterparts.

\subsection{Double chain-, Hansen-, and anti-blocking polytopes}
The \Defn{double chain polytope} associated to a double poset $\dP$ is the
polytope
\[
    \TChain{\dP} \ = \ \TChain{\P,\preceq_+,\preceq_-} \ := \ \conv \bigl\{
    (2\Chain{\P_+} \times \{1\} ) \cup (-2\Chain{\P_-} \times \{-1\} ) \bigr\}.
\]
The \Defn{reduced} version $\DChain{\dP} := \Chain{\P_+} - \Chain{\P_-}$
is studied in Section~\ref{sec:AB} in the context of \emph{anti-blocking}
polytopes. According to Fulkerson~\cite{Fulkerson}, a full-dimensional
polytope $\Po \subseteq \Rnn^n$ is \Defn{anti-blocking} if for any $q \in
\Po$, it contains all points $p \in \R^n$ with $0 \le p_i \le q_i$ for
$i=1,\dots,n$.  It is obvious from~\eqref{eqn:chain} that chain polytopes are
anti-blocking. Anti-blocking polytopes are important in combinatorial
optimization and, for example, contain stable set polytopes of graphs.  For
two polytopes $\Po_1,\Po_2 \subset \R^n$, we define the \Defn{Cayley sum} as
the polytope
\[
    \CaySum{\Po_1}{\Po_2} \ := \ \conv( \Po_1 \times \{1\} \cup \Po_2 \times
    \{-1\} )
\]
and we abbreviate $\CayDiff{\Po_1}{\Po_2} := \CaySum{\Po_1}{-\Po_2}$. Thus,
\[
    \TOrd{\dP} \ = \  \CayDiff{2\Ord{\P_+}}{2\Ord{\P_-}} \qquad \text{ and }
    \qquad
    \TChain{\dP} \ = \  \CayDiff{2\Chain{\P_+}}{2\Chain{\P_-}}. 
\]    
Section~\ref{sec:AB} is dedicated to a detailed study of the polytopes
$\CayDiff{\Po_1}{\Po_2}$ as well as their sections $\Po_1 - \Po_2$ for
anti-blocking polytopes $\Po_1,\Po_2 \subset \Rnn^n$. We completely determine
the facets of $\CayDiff{\Po_1}{\Po_2}$ in terms of $\Po_1,\Po_2$ in
Section~\ref{ssec:AB_Minkowski}, which yields the combinatorics of
$\TChain{\dP}$.  In Section~\ref{ssec:AB_subdiv}, we describe a canonical
subdivision of $\CayDiff{\Po_1}{\Po_2}$ and $\Po_1 - \Po_2$ for anti-blocking
blocking polytopes $\Po_1,\Po_2$. Moreover, if $\Po_1,\Po_2$ have regular,
unimodular, or flag triangulations, then so has $\CayDiff{\Po_1}{\Po_2}$
(for an appropriately chosen affine lattice).  The canonical subdivision
enables us to give explicit formulas for the volume and the Ehrhart
(quasi-)polynomials of these classes of polytopes.

The chain polytope $\Chain{\P}$ only depends on the comparability graph $G(\P)$
of $\P$ and, more precisely, is the stable set polytope of $G(\P)$. Thus,
$\TChain{\dP}$ only depends on the \Defn{double graph} $(G(\P_+),G(\P_-))$.
For a graph $G$, let $\Po_G$ be its stable set polytope; see
Section~\ref{ssec:AB_Minkowski} for precise definitions.
Lov\'{a}sz~\cite{Lovasz} characterized \emph{perfect} graphs in terms of
$\Po_G$ and Hansen~\cite{Hansen} studied the polytopes $\Hansen{G} :=
\CayDiff{2\Po_G}{2\Po_G}$. If $G$ is perfect, then Hansen showed that the polar
$\Hansen{G}^\dual$ is linearly isomorphic to $\Hansen{\overline{G}}$ where
$\overline{G}$ is the complement graph of $G$. In
Section~\ref{ssec:double_graph}, we generalize this result to all Cayley sums
of anti-blocking polytopes.

\subsection{$2$-level polytopes and volume}
A full-dimensional polytope $\Po \subset \R^n$ is called
\Defn{$\mathbf{2}$-level} if for any facet-defining hyperplane $H$ there is $t
\in \R^n$ such that $H \cup (t + H)$ contains all vertices of $\P$. The class
of $2$-level polytopes plays an important role in, for example, the study of
centrally-symmetric polytopes~\cite{WSZ,Hansen}, polynomial
optimization~\cite{GPT,gs14}, statistics~\cite{Sullivant}, and combinatorial
optimization~\cite{Schrijver}.  For
example, Lov\'{a}sz~\cite{Lovasz} characterizes perfect graphs by the
$2$-levelness of their stable set polytopes and Hansen showed that
$\Hansen{G}$ is $2$-level if $G$ is perfect. In fact, we extend this to yet
another characterization of perfect graphs in Corollary~\ref{cor:AB_2l}. This
result implies that $\TChain{\dP_\circ}$ is $2$-level for double posets induced
by posets. However, it is in general \emph{not} true that
$\CaySum{\Po_1}{\Po_2}$ is $2$-level if $\Po_1$ and $\Po_2$ are. A counterexample is the polytope
$\CaySum{\Delta_{6,2}}{\Delta_{6,4}}$, where $\Delta_{n,k}$ is the
$(n,k)$-hypersimplex. The starting point for this paper was the question for
which double posets $\dP$ the polytopes $\TOrd{\dP}$ and $\TChain{\dP}$ are $2$-level.
Answers are given in Corollary~\ref{cor:TO-2l},
Proposition~\ref{prop:2l-terti}, and Corollary~\ref{cor:AB_2l}. A new class of
$2$-level polytopes comes from valuation polytopes; see
Corollary~\ref{cor:val_2l}. Sullivant~\cite[Thm.~2.4]{Sullivant} showed that
$2$-level lattice polytopes $\Po$ have the interesting property that any
pulling triangulation that uses all lattice points in $\Po$ is unimodular.
Hence, for $2$-level lattice polytopes, the normalized volume is the number of
simplices.  In particular, $\Ord{\P}$ is $2$-level and Stanley's canonical
triangulation is a pulling triangulation. Stanley defined a piecewise linear
homeomorphism between $\Ord{\P}$ and $\Chain{\P}$ whose domains of linearity
are exactly the simplices of the canonical triangulation. Since this
\emph{transfer map} is lattice preserving, it follows that $\ehr_{\Ord{\P}}(n)
= \ehr_{\Chain{\P}}(n)$, which also implies the volume result. In
Section~\ref{sec:triang}, we generalize this transfer map to a lattice
preserving PL homeomorphism $\TPsi{\dP} : \R^\P \times \R \rightarrow \R^\P
\times \R$ for any compatible double poset $\dP$. In particular,
$\TChain{\dP}$ is mapped to $\TOrd{\dP}$. This also transfers the canonical
flag triangulation of $\TChain{\dP}$ to a canonical flag triangulation of
$\TOrd{\dP}$. Abstractly, the triangulation can be described in terms of a
suitable subcomplex of the order complex of the double Birkhoff lattice
$\TBirk{\dP} = \Birk{\P_+} \uplus \Birk{\P_-}$.  In
Section~\ref{ssec:TO_ehrhart}, we give explicit formulas for the Ehrhart
polynomial and the volume of $\TOrd{\dP}$ if $\dP$ is compatible and for
$\TChain{\dP}$ in general.

\subsection{Double Hibi rings}
Hibi~\cite{Hibi87} studied rings associated to finite posets that give posets
an algebraic incarnation and that are called \Defn{Hibi rings}. In modern
language, the Hibi ring $\Hring{\P}$ associated to a poset $\P$ is the
semigroup ring associated to $\Ord{\P}$. Many properties of $\P$ posses an
algebraic counterpart and, in particular, Hibi exhibited a quadratic Gr\"obner
basis for the associated toric ideal. In Section~\ref{sec:GB}, we introduce
the \Defn{double Hibi rings} $\dHring{\dP}$ as suitable analogs for double
posets, which are the semigroup rings associated to $\TOrd{\dP}$. We construct
a quadratic Gr\"obner basis for the cases of compatible double posets. Using a
result by Sturmfels~\cite[Thm.~8.3]{Sturmfels96}, this shows the existence of
a unimodular and flag triangulation of $\TOrd{\dP}$ which coincides with the
triangulation in Section~\ref{sec:triang}. We also construct a quadratic
Gr\"obner basis for the rings $\C[\TChain{\dP}]$ for arbitrary double posets
and we remark on the algebraic implications for double posets.

\textbf{Acknowledgements.} We would like to thank Christian Stump and Stefan
Felsner for many helpful conversations regarding posets and we thank Vic
Reiner for pointing out~\cite{MR}. We would also like to thank the referees
for valuable suggestions.  T.~Chappell was supported by a Phase-I scholarship
of the Berlin Mathematical School.  T.~Friedl and R.~Sanyal were supported by
the DFG-Collaborative Research Center, TRR 109 ``Discretization in Geometry
and Dynamics''. T.~Friedl received additional support from the Dahlem Research
School at Freie Universit\"at Berlin. 

\tableofcontents

\section{Double order polytopes}\label{sec:TO}

\subsection{Order polytopes}
Let $(\P,\preceq)$ be a poset. We write $\Phat$ for the poset obtained from
$\P$ by adjoining a minimum $\Pbot$ and a maximum $\Ptop$. For an order
relation $a \prec b$, we define a linear form $\ell_{a,b} :  \R^\P \rightarrow
\R$ by
\[
    \ell_{a,b}(f) \ := \ f(a) - f(b)
\]
for any $f \in \R^\P$.  Moreover, for $a \in \P$, we define $\ell_{a,\Ptop}(f)
:= f(a)$ and $\ell_{\Pbot,a}(f) := -f(a)$. With this notation, $f \in \R^\P$
is contained in $\Ord{\P}$ if and only if
\begin{equation}\label{eqn:ord_poly}
\begin{aligned}
        \ell_{a,b}(f) & \ \le \ 0 \quad \text{ for all } a \prec b,\\
        \ell_{\Pbot,b}(f) & \ \le \ 0 \quad \text{ for all } b  \in \P, \text{
        and}\\
        \ell_{a,\Ptop}(f) & \ \le \ 1 \quad \text{ for all } a  \in \P.
\end{aligned}
\end{equation}
Every nonempty face $F$ of $\Ord{\P}$ is of the form
\[
    F \ = \ \Ord{\P}^\ell \ := \ \{ f \in \Ord{\P} : \ell(f) \ge
    \ell(f') \text{ for all } f' \in \Ord{\P} \}
\]
for some linear function $\ell \in (\R^\P)^*$. Later, we want to identify
$\ell$ with its vector of coefficients and thus we write
\[
    \ell(f) \ = \ \sum_{a \in \P} \ell_a f(a).
\]

Combinatorially, faces can be described using \Defn{face partitions}: To every
face $F$ is an associated collection $B_1,\dots,B_m \subseteq \Phat$ of
nonempty and pairwise disjoint subsets that partition $\Phat$. According to
Stanley~\cite[Thm~1.2]{TwoPoset}, a partition of $\Phat$ is a (closed) face
partition if and only if each $(B_i, \preceq)$ is a connected poset and $B_i
\preceq' B_j :\Leftrightarrow p_i \preceq p_j$ for some $p_i \in B_i, p_j \in
B_j$ is a partial order on $\{B_1,\dots,B_m\}$. Of course, it is sufficient to
remember the non-singleton parts and we define the \Defn{reduced} face
partition of $F$ as $\FaPa(F) := \{ B_i : |B_i| > 1 \}$.  The \Defn{normal
cone} of a nonempty face $F \subseteq \Ord{\P}$ is the polyhedral cone
\[
    \NC_\P(F) \ := \ \{ \ell \in (\R^\P)^* : F \subseteq \Ord{\P}^\ell \}.
\]
The following description of $\NC_\P(F)$ follows directly
from~\eqref{eqn:ord_poly}.

\begin{prop}\label{prop:NC}
    Let $\P$ be a finite poset and  $F \subseteq \Ord{\P}$ a nonempty face
    with reduced face partition $\FaPa = \{B_1,\dots,B_k\}$. Then
    \[
        \NC_\P(F) \ = \ \cone \{\ell_{a,b}: [a,b] \subseteq B_i \text{ for
        some }i=1,\dots,k \}.
    \]
\end{prop}

We note the following simple but very useful consequence of this
description. 

\begin{cor}\label{cor:maxmin}
    Let $F \subseteq \Ord{\P}$ be a nonempty face with reduced face partition
    $\FaPa = \{ B_1,\dots,B_k \}$. Then for every $\ell \in \relint 
    \NC_\P(F)$ and $p \in \P$ the following hold: 
    \begin{compactenum}[\rm (i)]
    \item if $p \in \min(B_i)$ for some $i$, then $\ell_p > 0$;
    \item if $p \in \max(B_i)$ for some $i$, then $\ell_p < 0$;
    \item if $p \not\in \bigcup_i B_i$, then $\ell_p = 0$.
    \end{compactenum}
\end{cor}

The vertices of $\Ord{\P}$ are exactly the indicator functions $\1_\Filter :
\P \rightarrow \{0,1\}$ where $\Filter \subseteq \P$ is a filter.  For a
filter $\Filter \subseteq \P$, we write $\widehat{\Filter} := \Filter \cup
\{\Ptop\}$ for the filter induced in $\Phat$. 

\begin{prop}\label{prop:filter_fapa}
    Let $F \subseteq \Ord{\P}$ be a face with (reduced) face partition $\FaPa
    = \{B_1,\dots,B_k\}$ and let  $\Filter\subseteq \P$ be a filter. Then
    $\1_\Filter\in F$ if and only if
    \[
    \widehat{\Filter} \cap B_i \ = \ \emptyset \quad \text{ or } \quad
    \widehat{\Filter} \cap B_i \ = \ B_i
    \]
    for all $i=1,\dots,k$.
\end{prop}
That is, $\1_\Filter$ belongs to $F$ if and only if $\widehat{\Filter}$ does not separate
any two comparable elements in $B_i$, for all $i$. 

\subsection{Facets of double order polytopes} \label{ssec:TO_facets}
Let $\dP = (\P,\preceq_\pm)$ be a double poset. The double order polytope
$\TOrd{\dP}$ is a $(|\P|+1)$-dimensional polytope in $\R^\P \times
\R$ with coordinates $(f,t)$.  It is obvious that the vertices of
$\TOrd{\dP}$ are exactly $(2\1_{\Filter_+},1), (-2\1_{\Filter_-},-1)$
for filters $\Filter_+ \subseteq \P_+$ and $\Filter_- \subseteq \P_-$,
respectively.  To get the most out of our notational convention, for
$\sigma \in \{-,+\}$ we define
\[
    -\sigma \ := \
    \begin{cases}
        - & \text{ if } \sigma = +\\
        + & \text{ if } \sigma = -.
    \end{cases}
\]

By construction, $2\Ord{\P_+} \times \{1\}$ and $-2\Ord{\P_-} \times \{-1\}$
are facets that are obtained by maximizing the linear function $\pm
L_\emptyset(f,t) := \pm t$ over $\TOrd{\dP}$. We call the remaining facets
\Defn{vertical}, as they are of the form $\CayDiff{F_+}{F_-}$, where $F_\sigma
\subset \Ord{\P_\sigma}$ are certain nonempty proper faces for $\sigma = \pm$.
The vertical facets are in bijection with the facets of the reduced double
order polytope $\DOrd{\dP} = \Ord{\P_+} - \Ord{\P_-}$.

More precisely, if $F \subset \TOrd{\dP}$ is a facet, then there is a
linear function $\ell \in (\R^\P)^*$ such that $F = \CayDiff{F_+}{F_-}$ where
$F_+ = \Ord{\P_+}^\ell$ and $F_- = \Ord{\P_-}^{-\ell}$. This linear function
is necessarily unique up to scaling and hence the faces $F_+,F_-$ are
characterized by the property
\begin{equation}\label{eqn:rigid}
    \relint \NC_{\P_+}(F_+)  \ \cap \
    \relint -\NC_{\P_-}(F_-) \ = \ \R_{>0} \cdot \ell \, .
\end{equation}
We will call a linear function $\ell$ \Defn{rigid} if it
satisfies~\eqref{eqn:rigid} for a pair of faces $(F_+,F_-)$.  Our next goal is
to give an explicit description of all rigid linear functions for
$\TOrd{\dP}$ which then yields a characterization of vertical facets.

An \Defn{alternating chain} $C$ of a double poset $\dP = (P,\preceq_\pm)$ is a
finite sequence of distinct elements
\begin{equation}\label{eqn:alt_chain}
    \Pbot \ = \ 
    p_0 \ \prec_{\sigma} \
    p_1 \ \prec_{-\sigma} \
    p_2 \ \prec_{\sigma} \
    \cdots \ \prec_{\pm\sigma} \
    p_k \ = \ \Ptop,
\end{equation}
where $\sigma \in \{\pm\}$. For an alternating chain $C$, we define a linear
function $\ell_C$ by 
\[
    \ell_C(f) \ := \ \sigma \, \sum_{i=1}^{k-1} (-1)^i f(p_i).
\]
Here, we severely abuse notation and interpret $\sigma$ as $\pm 1$.  Note that
$\ell_C \equiv 0$ if $k=1$ and we call $C$ a \Defn{proper} alternating chain
if $k > 1$.  An \Defn{alternating cycle} $C$ of $\dP$ is a sequence of length
$2k$ of the form
\[
    p_0 \ \prec_{\sigma} \ p_1 \ \prec_{-\sigma} \ p_2 \ \prec_{\sigma} \
    \cdots \ \prec_{-\sigma} \ p_{2k} \
    = \ p_0,
\]
where $\sigma \in \{\pm\}$ and $p_i \neq p_j$ for $0 \le i < j < 2k$. We
similarly define a linear function associated to $C$ by
\[
    \ell_C(f) \ := \ \sigma \, \sum_{i=0}^{2k-1} (-1)^i f(p_i).
\]
Note that it is possible that a sequence of elements $p_1,p_2,\dots,p_k$ gives
rise to two alternating chains, one starting with $\prec_+$ and one starting
with $\prec_-$.  On the other hand, every alternating cycle of length $2k$
yields $k$ alternating cycles starting with $\prec_+$ and $k$ alternating
starting with $\prec_-$.

\begin{prop}\label{prop:rigid}
    Let $\dP=(\P,\preceq_\pm)$ be a double poset. If $\ell$ is a rigid linear
    function for $\TOrd{\dP}$, then $\ell = \mu \ell_C$ for some
    alternating chain or alternating cycle $C$ and $\mu > 0$.
\end{prop}
\begin{proof}
    Let $F_+ = \Ord{\P_+}^{\ell}$ and $F_- = \Ord{\P_-}^{-\ell}$ be the two
    faces for which~\eqref{eqn:rigid} holds and let $\FaPa_\pm = \{B_{\pm1},
    B_{\pm2},\dots\}$ be the corresponding reduced face partitions.  We define
    a directed bipartite graph $G = (V_+\cup V_-,E)$ with nodes $V_+ = \{ p
    \in \P : \ell_p > 0 \}$ and $V_-$ accordingly. If $\Ptop$ is contained in
    some part of $\FaPa_+$, then we add a corresponding node $\Ptop_+$ to
    $V_-$ Consistently, we add a node $\Ptop_-$ to $V_+$ if $\Ptop$ it occurs
    in a part of $\FaPa_-$.  Note that $\Pbot_-$ and $\Pbot_+$ are distinct
    nodes.  Similarly we add $\Pbot_{+}$ to $V_+$ and $\Pbot_{-}$ to $V_-$ if
    they appear in $\FaPa_+$ and $\FaPa_-$, respectively.  By
    Corollary~\ref{cor:maxmin}, we have ensured that $\max(B_{+i}) \subseteq
    V_-$ and $\max(B_{-i}) \subseteq V_+$ for all $i$.
    
    For $u \in V_+$ and $v \in V_-$, we add the directed edge $uv \in E$ if $u
    \prec_+ v$ and $[u,v]_{\P_+} \subseteq B_{+i}$ for some $i$. Similarly, we
    add the directed edge $vu \in E$ if $v \prec_- u$ and $[v,u]_{\P_-}
    \subseteq B_{-i}$ for some $i$. We claim that every node $u$ except for
    maybe the special nodes $\Pbot_{\pm},\Ptop_{\pm}$ has an incoming and an
    outgoing edge. For example, if $u \in V_+$, then $\ell_u > 0$. By
    Corollary~\ref{cor:maxmin}(iii), there is an $i$ such that $u \in B_{+i}$
    and by (ii), $u$ is not a maximal element in $B_{+i}$.  Thus, there is
    some $v \in \max(B_{+i})$ with $u \prec_+ v$ and $uv$ is an edge.  It
    follows that every longest path either yields an alternating cycle or a
    proper alternating chain.

    For an alternating cycle $C = (p_0 \prec_+ \cdots \prec_- p_{2l})$, we
    observe that 
    \begin{align*}
        \ell_C &\ = \ 
        \ell_{p_0,p_1} + 
        \ell_{p_2,p_3} + 
        \cdots + 
        \ell_{p_{2l-2},p_{2l-1}} \text{ and } \\
        -\ell_C &\ = \ 
        \ell_{p_1,p_2} + 
        \ell_{p_3,p_4} + 
        \cdots + 
        \ell_{p_{2l-1},p_{2l}}.
    \end{align*}
    Since for every $j$, $[p_{2j},p_{2j+1}]_{\P_+}$ is contained in some part
    of $\FaPa_+$, we conclude that $\ell_C \in \NC_{\P_+}(F_+)$.  Similarly,
    for all $j$, $[p_{2j-1},p_{2j}]_{\P_-}$ is contained in some part of
    $\FaPa_-$, and hence $-\ell_C \in \NC_{\P_-}(F_-)$. Assuming that $\ell$
    is rigid then shows that $\ell = \mu \ell_C$ for some $\mu > 0$.

    If $G$ does not contain a cycle, then let $C = (p_0,p_1,\dots,p_k)$ be a
    longest path in $G$. In particular $p_0 = \Pbot_{\pm}$ and $p_k =
    \Ptop_{\pm}$. The same reasoning applies and shows that $\ell_C \in
    \NC_{\P_+}(F_+) \cap -\NC_{\P_-}(F_-)$ and hence $\ell = \mu
    \ell_C$ for some $\mu > 0$. 
\end{proof}

In general, not every alternating chain or cycle gives rise to a rigid linear
function. Let $(\P,\preceq)$ be a poset that is not the antichain and define
the double poset $\dP = (\P,\preceq,\preceq^\op)$, where $\preceq^\op$ is the \Defn{opposite} order. In this case
$\TOrd{\dP}$ is, up to a shear, the ordinary prism over
$\Ord{\P,\preceq}$. Hence, the vertical facets of $\TOrd{\dP}$ are prisms over
the facets of $\Ord{\P}$. It follows from~\eqref{eqn:ord_poly} that these
facets correspond to cover relations in $\P$. Hence, every rigid $\ell$ is of
the form $\ell = \mu \ell_{p,q}$ for cycles $p \prec_+ q \prec_- p$ where $p
\prec q$ is a cover relation in $\P$.

We call a double poset $\dP = (\P,\preceq_+,\preceq_-)$ \Defn{compatible} if
$\P_+ = (\P,\preceq_+)$ and $\P_- = (\P,\preceq_-)$ have a common linear extension.  Note
that a double poset is compatible if and only if it does not contain
alternating cycles. Following~\cite{MR}, we call a double poset $\dP$
\Defn{special} if $\preceq_-$ is a total order. At the other extreme, we call
$\dP$ \Defn{anti-special} if $(\P,\preceq_-)$ is an anti-chain. A \Defn{plane
poset}, as defined in~\cite{Foissy1} is a double poset $\dP =
(\P,\preceq_+,\preceq_-)$ such that distinct $a,b \in \P$ are
$\prec_+$-comparable if and only if they are not $\prec_-$-comparable.  For
two posets $(\P_1,\preceq^1)$ and $(\P_2,\preceq^2)$ one classically defines
the \Defn{disjoint union} $\preceq_\uplus$ and the \Defn{ordinal sum}
$\preceq_\oplus$ as the posets on $\P_1 \uplus \P_2$ as follows. For $a,b \in
\P_1 \uplus \P_2$ set $a \preceq_{\uplus} b$ if $a,b \in \P_i$ and $a
\preceq^i b$ for some $i \in \{1,2\}$. For the ordinal sum, $\preceq_\oplus$
restricts to $\preceq^1$ and $\preceq^2$ on $\P_1$ and $\P_2$ respectively and
$p_1 \prec_\oplus p_2$ for all $p_1 \in \P_1$ and $p_2 \in \P_2$. The effect
on order polytopes is $\Ord{\P_1 \uplus \P_2} = \Ord{\P_1} \times \Ord{\P_2}$
and $\Ord{\P_1 \oplus \P_2}$ is a \emph{subdirect sum} in the sense of
McMullen~\cite{mcmullen}. Malvenuto and Reutenauer~\cite{MR} define the
\Defn{composition} of two double posets $(\P_1,\prec^1_\pm)$ and
$(\P'_2,\prec^2_\pm)$ as the double poset $(\P,\preceq_\pm)$ such that
$(\P,\preceq_+) = (\P_1,\prec^1_+) \uplus (\P_2,\prec^2_+)$ and
$(\P,\preceq_-) = (\P_1,\prec^1_-) \oplus (\P_2,\prec^2_-)$.

The following is easily seen; for plane posets with the help of~\cite[Prop.~11]{Foissy1}.
\begin{prop}\label{prop:compat}
    Anti-special and  plane posets are compatible. Moreover, the
    composition of two compatible double posets is a compatible double poset.
\end{prop}

This defining property of compatible double posets assures us that in an
alternating chain $p_i \prec_\sigma p_j$ implies $i < j$ for any $\sigma \in
\{\pm\}$. In particular, a compatible double poset does not have 
alternating cycles. This also shows the following.

\begin{lem}\label{lem:compat}
    Let $\dP = (\P,\preceq_\pm)$ be a compatible double poset.  If $a_{i}
    \prec_\sigma a_{i+1} \prec_{-\sigma} \cdots \prec_{-\tau} a_j \prec_\tau
    a_{j+1}$ is part of an alternating chain with $\sigma,\tau \in \{\pm\}$
    and $i < j$, then there is no $b \in \P$ such that $a_i \prec_\sigma b
    \prec_\sigma a_{i+1}$ and $a_j \prec_\tau b \prec_\tau a_{j+1}$.
\end{lem}

For compatible double posets, we can give complete characterization of facets.

\begin{thm}\label{thm:compat_facets}
    Let $\dP$ a compatible double poset. A linear function $\ell$ is rigid if
    and only if $\ell \in \R_{>0} \ell_C$ for some alternating chain $C$. In
    particular, the facets of $\TOrd{\dP}$ are in bijection with alternating
    chains.
\end{thm}
\begin{proof}
    We already observed that $2\Ord{\P_+} \times \{1\}$ and $-2\Ord{\P_-}
    \times \{-1\}$ correspond to the improper alternating chains $\Pbot
    \prec_\sigma \Ptop$ for $\sigma = \pm$.  By Proposition~\ref{prop:rigid}
    it remains to show that for any proper alternating chain $C$ the function
    $\ell_C$ is rigid. We only consider the case that $C$ is an alternating
    chain of the form
    \[
        \Pbot \ = \ p_0 \ \prec_+ \  p_1 \ \prec_- \  p_2 \ \prec_+ \
        \cdots \ \prec_+ \ p_{2k-1} \ \prec_- \  p_{2k} \ \prec_+ \  p_{2k+1}
        \ = \ \Ptop.
    \]
    The other cases can be treated analogously. 
    Let $F_+ = \Ord{\P_+}^{\ell_C}$ and and $F_- = \Ord{\P_-}^{-\ell_C}$ be
    the corresponding faces with reduced face partitions $\FaPa_\pm$.
    Define $O = \{ p_1, p_3, \dots, p_{2k-1}\}$ and $E =
    \{p_2,p_4,\dots,p_{2k}\}$.  Then for any set $A \subseteq \P$, we observe
    that $\ell_C(\1_A) = |E \cap A| - |O \cap A|$. If $\Filter$ is a filter of
    $\P_+$, then $p_{2i} \in \Filter$ implies $p_{2i+1} \in \Filter$ and hence
    $\ell_C(\1_\Filter) \le 1$ and thus $\1_\Filter \in F_+$ if and only if
    $\Filter$ does not separate $p_{2j}$ and $p_{2j+1}$ for $1 \le j \le k$.
    Likewise, a filter $\Filter \subseteq \P_-$ is contained in $F_-$ if and
    only if $\Filter$ does not separate $p_{2j-1}$ and $p_{2j}$ for $1 \le j
    \le k$.  Lemma~\ref{lem:compat} implies that 
    \begin{align*}
        \FaPa_+ &\ = \
        \{[p_0,p_1]_{\P_+}, [p_2,p_3]_{\P_+}, \dots, [p_{2k},
        p_{2k+1}]_{\P_+}\}\text{ and}\\
        \FaPa_- &\ = \ \{[p_1,p_2]_{\P_-}, [p_3,p_4]_{\P_-},
        \dots, [p_{2k-1},p_{2k}]_{\P_-} \}.
    \end{align*}
    To show that $\ell_C$ is rigid pick a linear function $\ell(\phi) =
    \sum_{p \in \P} \ell_p \phi(p)$ with $F_+ = \Ord{\P_+}^{\ell}$ and $F_- =
    \Ord{\P_-}^{-\ell}$. Since the elements in $E$ and $O$ are exactly the
    minimal and maximal elements of the parts in $\FaPa_+$, it follows from
    Corollary \ref{cor:maxmin} that $\ell_p > 0$ if $p \in E$, $\ell_p < 0$
    for $p \in O$. By Lemma~\ref{lem:compat}, it follows that if $q \in
    (p_i,p_{i+1})_{\P_+}$, then $q$ is not contained in a part of the reduced
    face partition $\FaPa_-$ and vice versa. By
    Corollary~\ref{cor:maxmin}(iii), it follows that $\ell_p = 0$ for $p
    \not\in E \cup O$. Finally, $\ell_{p_i} + \ell_{p_{i+1}} = 0$ for all $1
    \le i \le 2k$ by Proposition~\ref{prop:NC} and therefore $\ell = \mu
    \ell_C$ for some $\mu > 0 $ finishes the proof.
\end{proof}

\begin{example}
    Let $\dP = (\P,\preceq_\pm)$ be a compatible double poset with $|\P| =
    n$.
    \begin{enumerate}
        \item Let $\preceq_+ = \preceq_- = \preceq$ and $(\P,\preceq)$ be the
            $n$-antichain. Then the only alternating chains are of the form
            $\Pbot \prec_\sigma a \prec_{-\sigma} \Ptop$  for $a \in \P$. The
            double order polytope $\TOrd{\dP}$ is the $(n+1)$-dimensional cube
            with vertices $\{0,2\}^n \times \{+1\}$ and 
            $\{0,-2\}^n \times \{-1\}$.
        \item If $\preceq_+ = \preceq_- = \preceq$ and $(\P,\preceq)$ is the
            $n$-chain $[n]$, then any alternating chain can be identified with
            an element in $\{-,+\}^{n+1}$. More precisely,
            $\TOrd{\dP}$ is linearly isomorphic to the
            $(n+1)$-dimensional crosspolytope $C_{n+1}^\triangle = \conv\{ \pm
            \e_1,\dots, \pm \e_{d+1}\}$.
        \item Let $(\P,\preceq_+)$ be the $n$-chain and $(\P,\preceq_-)$ be
            the $n$-antichain. Then any alternating chain is of the form
            $\Pbot \prec_\sigma a \prec_{-\sigma} \Ptop$ for $\sigma = \pm$
            and any relation $a \prec_+ b$ be can be completed to a unique
            alternating chain. Thus, $\TOrd{\dP}$ is a
            $(n+1)$-dimensional polytope with $2^n + n + 1$ vertices and $
            \binom{n}{2} + 2n + 2$ facets.
        \item The \Defn{comb} (see Figure \ref{fig:comb}) is the poset $C_n$ on elements
            $\{a_1,\dots,a_n,b_1,\dots,b_n\}$ such that $a_i \preceq a_j$ if
            $i \le j$ and $b_i \prec a_i$ for all $i,j \in [n]$. The $n$-comb
            has $2^{n+1}-1$ filters and $3\cdot2^{n} - 2$ chains. Hence
            $\TOrd{C_n,\preceq,\preceq}$ has $2^{n+2}-2$ vertices and
            $3\cdot2^{n+1} - 4$ facets.
        \item Generally, let $\P_1, \P_2$ be two posets and denote by $f_i$
            and $c_i$ the number of filters and chains of $\P_i$ for $i=1,2$.
            Let $\dP_\circ$ be the trivial double poset induced by $\P_1 \uplus
            \P_2$. Then $\TOrd{\dP_\circ}$ has $2f_1f_2$ vertices and
            $2(c_1+c_2) - 2$ facets.
    \end{enumerate}
\end{example}

\begin{example}
    Consider the compatible 'XW'-double poset $\dP_{XW}$ on five elements,
    whose Hasse diagrams are given in Figure \ref{fig:XW}. The polytope
    $\TOrd{\dP_{XW}}$ is six-dimensional with face vector
    \[
        f(\TOrd{\dP_{XW}}) \ = \ (21, 112, 247, 263, 135, 28).
    \]
    The facets correspond to the 28 alternating chains in
    $\widehat\dP_{XW}$, which are shown in Figure \ref{fig:XW_alt}.
\end{example}

\begin{figure}[h]
  \centering
  \begin{minipage}{.49\textwidth}
  \centering
    \includegraphics[width=0.4\textwidth]{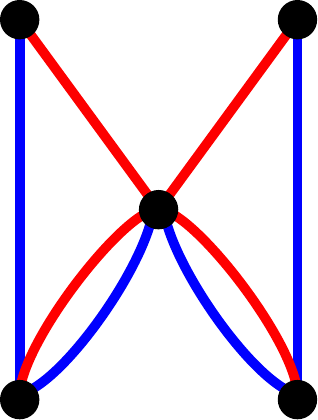}
    \caption{The 'XW'-double poset $\dP_{XW}$. The red and blue lines are the Hasse diagram of $\P_+$ and $\P_-$, respectively. Striped lines are edges in both Hasse diagrams.}
    \label{fig:XW}
    \end{minipage}
      \begin{minipage}{.49\textwidth}
  \centering
    \includegraphics[width=0.4\textwidth]{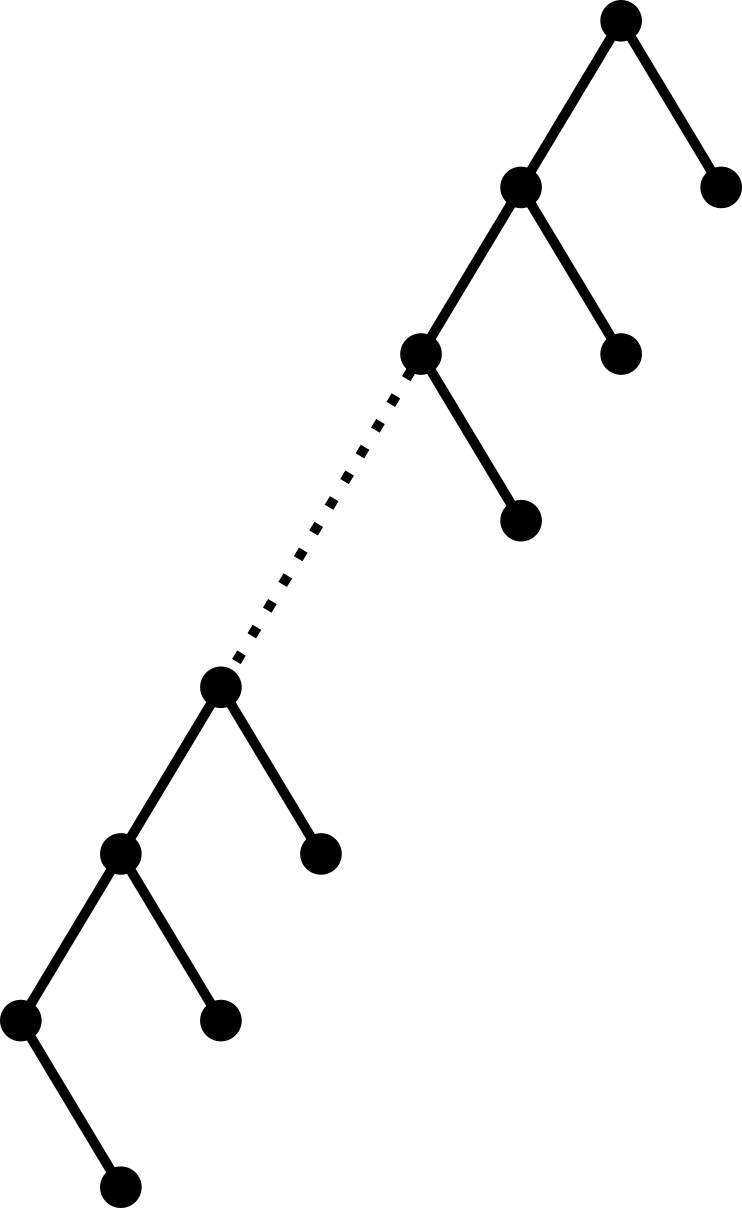}
    \caption{The comb $C_n$.}
    \label{fig:comb}
    \end{minipage}
\end{figure}

\begin{figure}[h]
  \centering
    \includegraphics[width=0.6\textwidth]{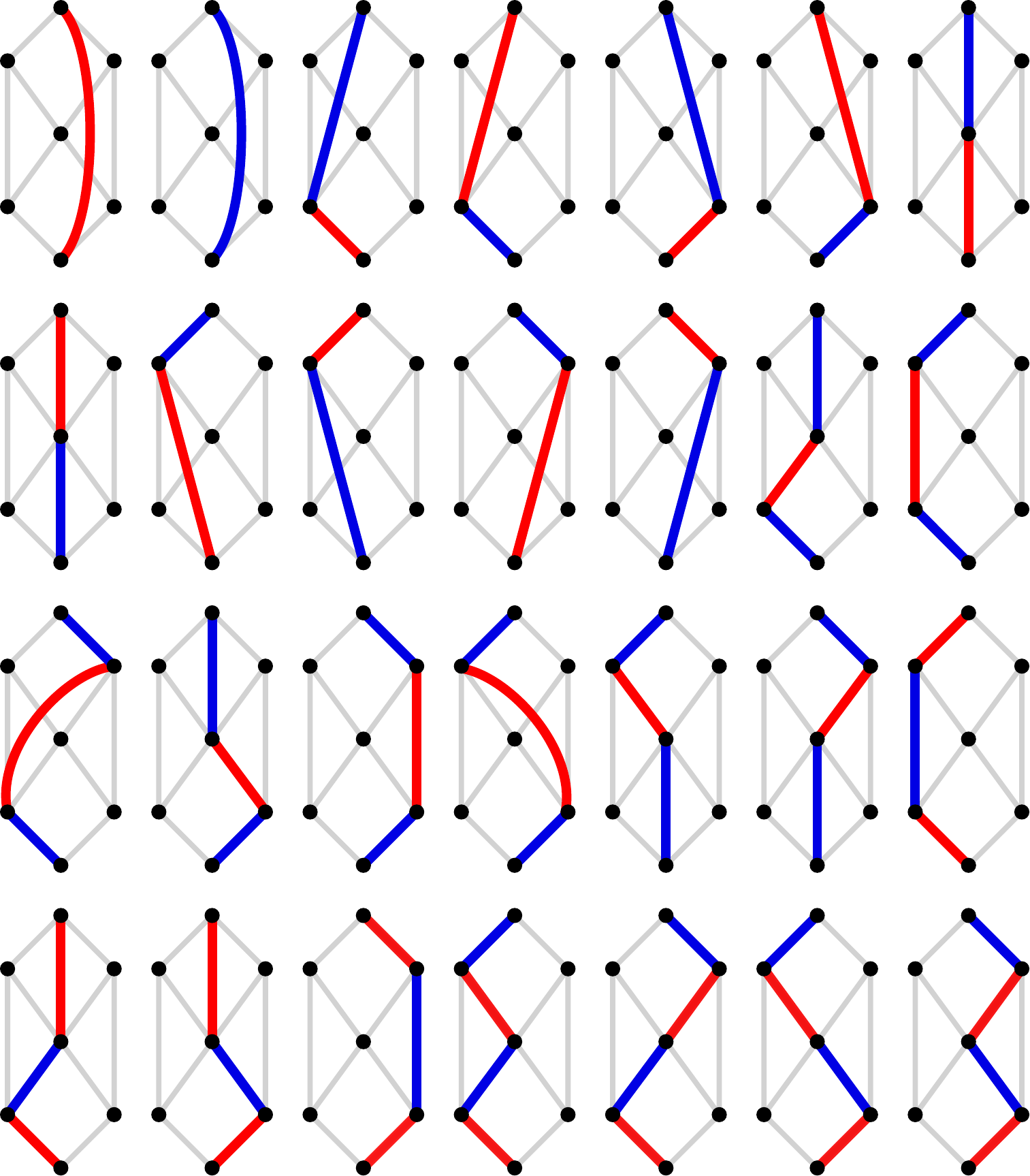}
    \caption{The 28 alternating chains in $\widehat \dP_{XW}$.}
    \label{fig:XW_alt}
\end{figure}

For two particular types of posets, we wish to determine the combinatorics of
$\TOrd{\dP}$ in more detail.

\begin{example}[Dimension-$2$ posets] \label{ex:2dim}
    For $n \ge 1$, let $\pi = (\pi_1,\pi_2,\dots,\pi_n)$ be an ordered
    sequence of distinct numbers. We may define a partial order $\preceq_\pi$
    on $[n]$ by $i \prec_\pi j$ if $i < j$ and $\pi_i < \pi_j$. Following
    Dushnik and Miller~\cite{Dushnik}, these are, up to isomorphism, exactly
    the posets of order dimension $2$.  A chain in $\P_\pi :=
    ([n],\preceq_\pi)$ is a sequence $i_1 < i_2 < \cdots < i_k$ with $
    \pi_{i_1} < \pi_{i_2} < \cdots < \pi_{i_k}$. Thus, chains in $\P_\pi$ are
    in bijection to \Defn{increasing subsequences} of $\pi$.  Conversely, one
    checks that filters (via their minimal elements) are in bijection to
    decreasing subsequences. It follows from Theorem~\ref{thm:compat_facets}
    that facets and vertices of $\TOrd{[n],\preceq_\pi,\preceq_\pi}$ are in
    2-to-1 correspondence with increasing and decreasing sequences,
    respectively.
\end{example}

\begin{example}[Plane posets]
    Let $\dP = (\P,\preceq_+,\preceq_-)$ be a compatible double poset.  We may
    assume that $\P = \{ p_1,\dots, p_n\}$ are labelled such that $p_i
    \prec_\sigma p_j$ for $\sigma = +$ or $= -$ implies $i < j$.
    By~\cite[Prop.~15]{Foissy2}, $\dP$ is a plane poset, if and only if  there
    is a sequence of distinct numbers $\pi = (\pi_1,\pi_2,\dots,\pi_n)$ such
    that for $p_i,p_j \in \P$
    \[
        \begin{aligned} 
            p_i \prec_+ p_j & \quad \Longleftrightarrow  \quad i < j \text{
            and } \pi_i < \pi_j \text{ and }\\
            p_i \prec_- p_j & \quad \Longleftrightarrow  \quad i < j \text{
            and } \pi_i > \pi_j.
        \end{aligned}
    \]
    This is to say, $\P_+$ is canonically isomorphic to $([n], \preceq_\pi)$ and
    $\P_-$ is canonically isomorphic to $([n],\preceq_{-\pi})$. It
    follows that alternating chains in $\dP$ are in bijection to
    \Defn{alternating sequences}. That is, sequences $i_1 < i_2 < i_3 < \cdots
    < i_k$ such that $\pi_{i_1} < \pi_{i_2} > \pi_{i_3} < \cdots$. Hence, by
    Theorem~\ref{thm:compat_facets}, the facets of $\TOrd{\dP}$ are in
    bijection to alternating sequences of $\pi$ whereas the vertices are in
    bijection to increasing and decreasing sequences of $\pi$.
\end{example}

As a consequence of the proof of Theorem~\ref{thm:compat_facets} we can
determine a facet-defining inequality description of double order polytopes.
For an alternating chain $C$ as in~\eqref{eqn:alt_chain}, let us write
$\sgn(C) = \tau \in \{-,+\}$ if the last relation in $C$ is $p_{k-1} \prec_\tau p_k=\Ptop$.

\begin{cor}\label{cor:TO_ineq}
    Let $\dP = (\P,\preceq_\pm)$ be a compatible double poset. Then $\TOrd{\dP}$
    is the set of points $(f,t) \in \R^\P \times \R$ such that 
    \[
       L_C(f,t) \ := \ \ell_C(f) - \sgn(C)\, t \ \le \  1 
    \]
    for all alternating chains $C$ of $\dP$.
\end{cor}
\begin{proof}
    Note that $0$ is in the interior of $\TOrd{\dP}$. Hence by
    Theorem~\ref{thm:compat_facets} every facet-defining halfspace of
    $\TOrd{\dP}$ is of the form $\{ (\phi,t) : L(\phi,t) = \mu \ell_C +
    \beta t \le 1\}$ for some alternating chain $C$ and $\mu, \beta \in \R$
    with $\mu > 0$. If $C$ is an alternating chain with $\sgn(C) = +$, then
    the maximal value of $\ell_C$ over $2\Ord{\P_+}$ is $2$ and $0$ over
    $-2\Ord{\P_-}$. The values are exchanged for $\sgn(C) = -$. It then follows
    that $\mu = 1$ and $\beta = -\sgn(C)$.
\end{proof}

With this, we can characterize the $2$-level polytopes among compatible
double order polytopes.
\begin{cor}\label{cor:TO-2l}
    Let $\dP = (\P,\preceq_\pm)$ be a compatible double poset. Then
    $\TOrd{\dP}$ is $2$-level if and only if $\preceq_+ = \preceq_-$.  In
    this case, the number of facets of $\TOrd{\dP_\circ}$ is twice the
    number of chains in $(\P,\preceq)$.
\end{cor}
\begin{proof}
    If $\preceq_+ = \preceq_- = \preceq$, then every alternating chain is a
    chain in $\P$ and conversely, every chain in $(\P,\preceq)$ gives rise to
    exactly two distinct alternating chains in $(\P,\preceq_\pm)$. In this
    case, it is straightforward to verify that the minimum of $\ell_C$ over
    $2\Ord{\P}$ is $-2$ if $\sgn(C) = +$ and $0$ otherwise. The claim now
    follows from Corollary~\ref{cor:TO_ineq} and together with
    Theorem~\ref{thm:compat_facets} also yields the number of facets.

    The converse follows from Proposition~\ref{prop:2l-terti} by noting that
    if both $(\P,\preceq_+,\preceq_-)$ and $(\P,\preceq_-,\preceq_+)$ are
    compatible and tertispecial then $\preceq_+ = \preceq_-$.
\end{proof}

In~\cite{Grinberg} a double poset $(\P,\preceq_+,\preceq_-)$ is called
\Defn{tertispecial} if $a$ and $b$ are $\prec_{-}$-comparable whenever
$a\prec_+ b$ is a cover relation for $a,b \in \P$.

\begin{prop}\label{prop:2l-terti}
    Let $\dP = (\P,\preceq_\pm)$ be a double poset.  If $\TOrd{\dP}$ is
    $2$-level, then $\dP$ as well as $(\P,\preceq_-,\preceq_+)$ are
    tertispecial.
\end{prop}
\begin{proof}
    Let $\sigma = \pm$ and let $a \prec_\sigma b$ be a cover relation. The
    linear function $\ell_{a,b}$ is facet defining for $\Ord{\P_\sigma}$ and
    hence yields a facet for $\TOrd{\dP}$. If $a,b$ are not comparable in
    $\P_{-\sigma}$, then the filters $\emptyset, \{ c \in \P : c \succeq_-
    a\}$ and $\{ c \in \P : c \succeq_- b\}$ take three distinct values on
    $\ell_{a,b}$.
\end{proof}

Let us remark that the number of facets of a given double poset $\dP =
(\P,\preceq_+,\preceq_-)$ can be computed by the transfer-matrix method. Let us
define the matrices $\eta^+, \eta^- \in \R^{\Phat \times \Phat}$ by 
\[
    \eta^\sigma_{a,b} \ := \
    \begin{cases}
        1 & \text{ if } a \prec_\sigma b\\
        0 & \text{ otherwise}
    \end{cases}
\]
for $a,b \in \Phat$ and $\sigma = \pm$.  Then $(\eta^+\eta^-)^k_{\Pbot,\Ptop}$
is the number of alternating chains of $\dP$ of length $k$ starting with
$\prec_+$ and ending with $\prec_-$. This shows the following.

\begin{cor}\label{cor:num_facets}
    Let $\dP = (\P,\preceq_+,\preceq_-)$ be a compatible double poset. Then the
    number of facets of $\TOrd{\dP}$ is given by
    \[
    \left[ (\mathrm{Id} - \eta^+\eta^-)^{-1}(\mathrm{Id} + \eta^+) + 
    (\mathrm{Id} - \eta^-\eta^+)^{-1}(\mathrm{Id} + \eta^-) \right]_{\Pbot,\Ptop}.
    \]
\end{cor}

\subsection{Faces and embedded sublattices}\label{ssec:TO_faces}%
The \Defn{Birkhoff lattice} $\Birk{\P}$ of a finite poset $\P$ is the
distributive lattice given by the collection of filters of $\P$  ordered by
inclusion. A subposet $\L \subseteq \Birk{\P}$ is called an \Defn{embedded
sublattice} if for any two filters $\Filter,\Filter' \in \Birk{\P}$
\[
    \Filter \cup \Filter', \Filter \cap \Filter' \in \L \quad \text{if and
    only if} \quad \Filter, \Filter' \in \L.
\]
For a subset $L \subseteq \Birk{\P}$ of filters we write $F(L) :=
\conv(\1_\Filter : \Filter \in L)$. Embedded sublattices give an alternative
way to characterize faces of $\Ord{\P}$.

\begin{thm}[{\cite[Thm~1.1(f)]{Wagner}}] \label{thm:O_faces}
    Let $\P$ be a poset and $\L \subseteq \Birk{\P}$ a collection of filters.
    Then $F(\L)$ is a face of $\Ord{\P}$ if and only if $\L$ is an embedded
    sublattice.
\end{thm}

We will generalize this description to the case of double order polytopes.
Throughout this section, let $\dP = (\P,\preceq_+,\preceq_-)$ be a double
poset. We define $\TBirk{\dP} := \Birk{\P_+} \uplus \Birk{\P_-}$. For any
subset $\L \subseteq \TBirk{\dP}$ we will denote by $\L_+$ the set $\L \cap
\Birk{\P_+}$ and we define $\L_-$ accordingly. Moreover, we shall write
\newcommand\Tface[1]{\overline{F}(#1)}
\begin{equation}\label{eqn:Tfaces}
    \Tface{\L} \ := \  \conv \left( \{ (2\1_{\Filter_+}, +1) : \Filter_+ \in
    \L_+ \}
    \cup \{ (2\1_{\Filter_-}, -1) : \Filter_- \in \L_- \} \right) 
    \ \subseteq \ \TOrd{\dP}.
\end{equation}
Thus, $\Tface{\L}  = \CayDiff{2F(\L_+)}{2F(\L_-)}$.

\begin{thm}\label{thm:TO_faces}
    Let $\dP = (\P,\preceq_+,\preceq_-)$ be a compatible double poset and
    $\L \subseteq \TBirk{\dP}$. Then $\Tface{\L}$ is a face of
    $\TOrd{\dP}$ if and only if
    \begin{enumerate}[\rm (i)]
        \item $\L_+ \subseteq \Birk{\P_+}$ and $\L_- \subseteq \Birk{\P_-}$
            are embedded sublattices and
        \item for all filters $\Filter_\sigma \subseteq \Filter_\sigma' \in
            \Birk{\P_\sigma}$ for $\sigma = \pm$ such that 
            \[
                \Filter_+' \setminus \Filter_+ \ = \ \Filter_-' \setminus
                \Filter_-
            \]
            it holds that $\{\Filter_+,\Filter_-\} \subseteq \L$ if and only if 
            $\{\Filter_+',\Filter_-'\} \subseteq \L$.
    \end{enumerate}
\end{thm}

We call a pair $\L = \L_+ \uplus \L_- \subseteq \TBirk{\dP}$ of embedded
sublattice \Defn{cooperating} if they satisfy condition (ii) of
Theorem~\ref{thm:TO_faces} above.  We may also rephrase condition (ii) as
follows.

\begin{lem}\label{lem:cond2}
    Let $\L_\sigma \subset \Birk{\P_\sigma}$ be an embedded sublattice for
    $\sigma = \pm$.
     Then $\L_+, \L_-$
    are cooperating if and only if only if for any two filters 
    $\Filter_- \in \L_-, \Filter_+ \in \L_+$ the following holds:
    \begin{enumerate}[\rm (a)]
        \item For $A \subseteq \min(\Filter_+) \cap \min(\Filter_-)$ we have
            $\Filter_- \setminus A \in \L_- $  and 
            $\Filter_+ \setminus A \in \L_+ $, and  
    \item for $B \subseteq \max(\P_+ \setminus \Filter_+) \cap \max(\P_-
        \setminus \Filter_-)$ we have $\Filter_- \cup B \in \L_- $  and
        $\Filter_+ \cup B \in \L_+ $.
    \end{enumerate}
\end{lem}
\begin{proof}
    It follows from the definition that for sets as stated, condition (ii)
    implies $\Filter_\sigma \setminus A, \Filter_\sigma \cup B \in \L_\sigma$ 
    for $\sigma = \pm$. 
    For the converse direction, let $\Filter_\sigma \subseteq \Filter_\sigma'$
    such that $\Filter_\sigma' \in \L_\sigma$ for $\sigma = \pm$.  Assume that
    $D := \Filter_+' \setminus \Filter_+  = \Filter_-' \setminus \Filter_-$.
    Then $A : = \min(D) \subseteq \min(\Filter_+') \cap \min(\Filter_-')$ and
    by (a), $\Filter_\sigma' \setminus A \in \L_\sigma$ for $\sigma = \pm$ and
    induction on $|D|$ yields the claim.
\end{proof}

Theorem~\ref{thm:TO_faces} can be deduced from the description of facets in
Theorem~\ref{thm:compat_facets}. We will give an alternative proof using
Gr\"obner bases in Section~\ref{sec:GB}. In conjunction with
Theorem~\ref{thm:TO_GB}, we can read the dimension of $\Tface{\L}$ from the
cooperating pair $\L$.  In the case of order polytopes, the canonical
triangulation (see Section~\ref{sec:triang}) of $\Ord{\P}$ yields the
following.

\begin{cor}\label{cor:O_dim_faces}
    Let $F \subseteq \Ord{\P}$ be a face with corresponding embedded
    sublattice $\L \subseteq \Birk{\P}$. Then $\dim F = l(\L) - 1$ where
    $l(\L)$ is the length of a longest chain in $\L$.
\end{cor}

Let $\dP = (\P,\preceq_+,\preceq_-)$ be a double poset and let $C_\sigma
\subseteq \Birk{\P_\sigma}$ be a chain of filters in $(\P,\preceq_\sigma)$ for
$\sigma = \pm$.  The pair of chains $C = C_+ \uplus C_-$ is
\Defn{non-interfering} if $\min(\Filter_+) \cap \min(\Filter_-) = \emptyset$
for any $\Filter_+ \in C_+$ and $\Filter_- \in C_-$. For $\L \subseteq
\TBirk{\dP}$, we denote by  $cl(\L)$ the maximum over $|C| = |C_+| + |C_-|$
where $C \subseteq \L$ is a pair of non-interfering chains.

\begin{figure}[h]
    \centering
    \includegraphics[width=0.3\textwidth]{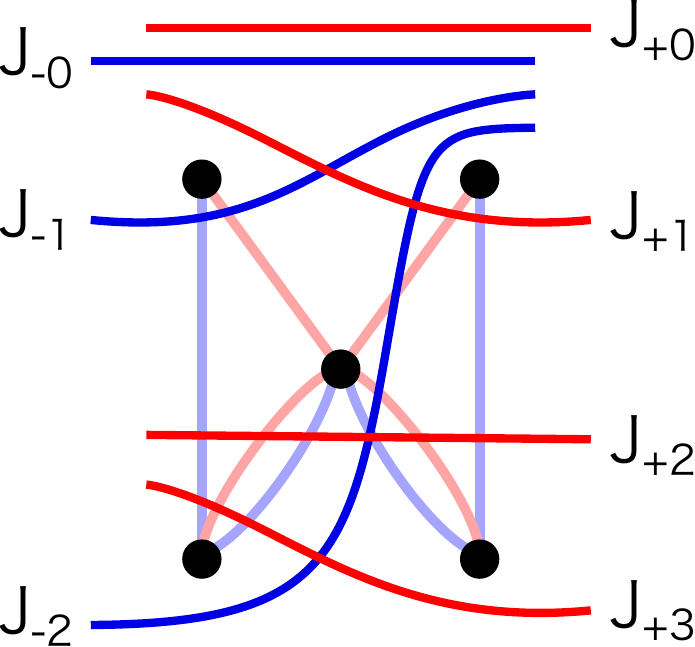}
    \caption{A maximal non-interfering set of filters in $\dP_{XW}$. A red or
    blue curve denotes the filter consisting of all elements above the curve.}
    \label{XW_fil}
\end{figure}

\begin{cor}\label{cor:TO_dim_face}
    Let $\dP$ be a compatible double poset and let $\L \subseteq
    \TBirk{\dP}$ be a cooperating pair of embedded sublattices. Then $\dim
    \Tface{\L} = cl(\L)-1$.
\end{cor}

As a consequence of Theorem~\ref{thm:O_faces}, 
$[\1_{\Filter},\1_{\Filter'}] \subseteq \Ord{\P}$ is an edge if and only if
$\Filter \subseteq \Filter'$ are filters of $\P$ such that $\Filter' \setminus
\Filter$ is a connected poset. Of course, this description captures all the
\emph{horizontal} edges of $\TOrd{\dP}$. The upcoming characterization
of vertical edges follows from Theorem~\ref{thm:compat_facets} but we
supply a direct proof.

\begin{cor}\label{cor:TO_vertical_edges}
    Let $\dP$ be a compatible double poset and let $\Filter_+ \subseteq \P_+$
    and $\Filter_- \subseteq \P_-$ be filters.  Then $(2\1_{\Filter_+},+1)$
    and $(-2\1_{\Filter_-},-1)$ are the endpoints of a vertical edge of
    $\TOrd{\dP}$ if and only if $\1_{\Filter_+} - \1_{\Filter_-}$ is a
    vertex of $\DOrd{\dP}$ if and only if 
    \[
        \min(\Filter_+) \cap \min(\Filter_-) \ = \ \emptyset \quad \text{
        and}\quad
        \max(\P_+ \setminus \Filter_+) \cap \max(\P_- \setminus \Filter_-) \
        = \ \emptyset.
    \]
\end{cor}
\begin{proof}
    The first equivalence follows from the fact that 
    \[
        \TOrd{\dP} \cap \{(\phi,t) : t = 0\} \ = \ 
        (\Ord{\P_+}-\Ord{\P_-})  \times \{0\}
    \]  
    and $\1_{\Filter_+} - \1_{\Filter_-}$ is the midpoint between
    $(2\1_{\Filter_+},+1)$ and $(-2\1_{\Filter_-},-1)$.

    Before we come to the second claim, let us note that the face partition of
    a vertex $\1_\Filter$ for a poset $(\P,\preceq)$ is given by $\{ \Filter,
    \P \setminus \Filter\}$.  Thus, if $\1_{\Filter_+} - \1_{\Filter_-}$ is a
    vertex of $\DOrd{\dP}$, then there is a linear function $\ell(f) = \sum_{a
    \in P} \ell_a f(a)$ such that $\Ord{\P_+}^\ell = \{\1_{\Filter_+}\}$ and
    $\Ord{\P_-}^{-\ell} = \{\1_{\Filter_-}\}$. Corollary~\ref{cor:maxmin} then
    yields that $\ell_a > 0$ for each $a \in \min(\Filter_+)$ and $\ell_a < 0$
    for $a \in \min(\Filter_-)$. The same reasoning applies to $\max(\P_+
    \setminus \Filter_+)$ and $\max(\P_- \setminus \Filter_-)$ and shows
    necessity.

    Let $b \in \min(\Filter_+)$. If $b \not \in \Filter_-$, then the linear
    function $\ell(f) := f(b)$ is maximized over $\Ord{\P_+}$ at every filter
    that contains $b$ and over $-\Ord{\P_-}$ at every filter that does not
    contain $b$.  If $b \in \Filter_-$, then, by assumption, $b \not\in
    \min(\Filter_-)$ and there is some $p_2 \in \min(\Filter_-)$ with $p_2
    \prec_- b$. Now, if $p_2 \in \Filter_+$, then there is $p_3 \in
    \min(\Filter_+)$ with $p_3 \prec_+ p_2$ and so on. Compatibility now
    assures us that we get a descending alternating chain of the form 
    \[
       \Ptop  \ \succ_+ \ b =: p_1 \ \succ_- \ p_{2} \ \succ_+ \ p_{3} \
       \succ_- \ \cdots \ \succ_{\sigma} \  p_k \ \succ_{-\sigma} \ a \
       \succ_\sigma \ \Pbot
    \]
    where $p_2, p_4, p_6,\ldots \in \min(\Filter_-) \cap \Filter_+$ and $p_1,
    p_3, p_5,\ldots \in \min(\Filter_+) \cap \Filter_-$ and
    $a\in\min(\Filter_{-\sigma})\setminus \Filter_\sigma$.  Consider the
    associated linear function 
    \[
        \ell(f)  \ = \ f(p_0) - f(p_1) + f(p_{2}) - \cdots + (-1)^{k} f(p_k) +
        (-1)^{k+1} f(a)
    \]
    for $f\in\R^\P$.  We claim that $\ell(\1_{\Filter'_+}) \le 1$  for each
    filter $\Filter'_+ \subseteq \P_+$ and with equality if $b \in
    \Filter'_+$. Indeed, if $p_{2i+1} \in \Filter'_+$, then $p_{2i} \in
    \Filter'_+$ for all $i \ge 1$. Conversely,
    $\ell(-\1_{\Filter'_-})  \le 0 = \ell(-\1_{\Filter_-})$ for each filter
    $\Filter'_- \subseteq \P_-$. This follows from the fact that $p_{2i}
    \in \Filter'_-$ implies $p_{2i-1} \in \Filter'_-$ for each $i \ge 1$.

    For $a \in \max(\P_+\setminus \Filter_+)$ the situation is similar and we
    search for $b \in \max(\P_- \setminus \Filter_-)$ with $a \prec_- b$ in
    the case that $a \not\in \Filter_-$. This yields a linear function $\ell
    \in -\NC_{\P_-}(\1_{\Filter_-})$ that is maximized over $\Ord{\P_+}$ at
    filters $\1_{\Filter'_+}$ with $a \not\in \Filter'_+$. Summing these
    linear functions for $b \in \min(\Filter_+)$ and $a \in \max(\P_+\setminus
    \Filter_+)$ yields a linear function $\ell^+$ with  $\Ord{\P_+}^{\ell^+} =
    \{ \1_{\Filter_+} \}$ and  $\1_{\Filter_-} \in \Ord{\P_-}^{-\ell^+}$.

    Of course, the same reasoning applies to $\Filter_-$ instead of
    $\Filter_+$ and it follows that $\ell^+ - \ell^-$ is uniquely maximized at
    $\1_{\Filter_+} - \1_{\Filter_-}$ over $\DOrd{\dP} = \Ord{\P_+} -
    \Ord{\P_-}$.
\end{proof}

\subsection{Polars and valuation polytopes}\label{ssec:val}

A real-valued \Defn{valuation} on a finite distributive lattice
$(\Birkhoff,\vee,\wedge)$ is a function $h : \Birkhoff \rightarrow \R$ such
that for any $a,b \in \Birkhoff$, 
\begin{equation}
    h(a \vee b) \ = \ h(a) + h(b) - h(a \wedge b)
\end{equation}
and $h(\Pbot) = 0$. Geissinger~\cite{Geissinger} studied the
\Defn{valuation polytope}
\[
    \Val(\Birkhoff) \ := \ \{ h : \Birkhoff \rightarrow [0,1] : h \text{
    valuation} \}
\]
and conjectured that its vertices are exactly the valuations with values in
$\{0,1\}$. This was shown by Dobbertin~\cite{Dobbertin}. Not much is known
about the valuation polytope and Stanley's
\emph{`5-'}-Exercise~\cite[Ex.~4.61(h)]{EC1new} challenges the reader to find
interesting combinatorial properties of $\Val(\Birkhoff)$.  In this section, we
prove a curious relation between valuation polytopes and order polytopes. 

It follows from Birkhoff's fundamental theorem (cf.~\cite[Sect.~3.4]{EC1new})
that any finite distributive lattice $\Birkhoff$ is of the form $\Birkhoff =
\Birkhoff(\P)$, that is, it is the lattice of filters of some poset $\P$. In
particular, for every valuation $h : \Birkhoff(\P) \rightarrow \R$ there is a
unique $h_0 : \P \rightarrow \R$ such that
\[
    h(\Filter) \ = \ \sum_{a \in \Filter} h_0(a),
\]
for every filter $\Filter \subseteq \P$. Hence, $\Val(\Birkhoff)$ is linearly
isomorphic to the $|\P|$-dimensional polytope
\[
    \ValP (\P) \ := \ \{ h_0 : \P \rightarrow \R : 0 \le
    h(\Filter) \le
    1 \text{ for all filters } \Filter \subseteq \P \}.
\]
We denote by
\[
    S^\dual \ = \ \{ \ell \in (\R^n)^* : \ell(s) \le 1 \text{ for all } s \in
    S \}
\]
the \Defn{polar} of a set $S \subset \R^d$.  For a polytope $\Po\subset\R^d$ we write
$\tprism{\Po}:=\CayDiff{\Po}{\Po}\subset\R^{d+1}$
for the \Defn{twisted prism} of $\Po$.

\begin{thm}\label{thm:twisted_val}
    For any finite poset $\P$
    \[
    \TOrd{\dP_\circ}^\dual \ = \ \TOrd{\P,\preceq,\preceq}^\dual \ = \
    \tprism{-\ValP(\P)}.
    \] 
\end{thm}

\begin{proof}
    For a chain $C = \{ a_0 \prec a_1 \prec \cdots \prec a_k \}$ in $\P$, we
    define
    \[
        \ell'_C(f) \ := \ \sum_{i=0}^k (-1)^{k-i} f(a_i)
    \]
    and $L'_C(f,t) :=  \ell'_C(f) - t$. It follows from
    Corollary~\ref{cor:TO_ineq} and Corollary~\ref{cor:TO-2l} that 
    \[
        \TOrd{\dP_\circ}^\dual \ = \ \conv ( \pm L'_C(f,t) : C \subseteq \P
        \text{ chain} ).
    \]
    It is shown in Dobbertin~\cite[Theorem~B]{Dobbertin} that 
    \[
        \ValP(\P) \ = \ \conv\left( \ell'_C :  C \subseteq
        \P \text{ chain} \right),
    \]
    from which the claim follows.
\end{proof}

As a direct consequence, we note the following.
\begin{cor}\label{cor:val_2l}
    Let $\P$ be a finite poset. Then $\tprism{\Val(\P)}$ is $2$-level.
\end{cor}
\begin{proof}
    Since $\TOrd{\dP_\circ}$ is centrally-symmetric and, by
    Corollary~\ref{cor:TO-2l}, $2$-level,  it follows that every vertex of
    $\TOrd{\dP_\circ}$ takes the values $+1$ or $-1$ on every facet-defining
    linear function. The vertices correspond to facet normals under polarity,
    which shows that $\TOrd{\dP_\circ}^\dual$ is $2$-level.
    Theorem~\ref{thm:twisted_val} now yields the claim.
\end{proof}
        
We can make the connection to valuations more transparent by considering
valuations with values in $[-1,1]$. Let $\Val^\pm(\Birkhoff(\P))$ denote the
corresponding polytope, then
\begin{equation}\label{eqn:ValPpm}
    \ValP^\pm(\P) \ = \ \{ h_0 : \P \rightarrow \R : -1 \le
    h(\Filter) \le 1 \text{ for all filters } \Filter \subseteq \P \}
    \ = \ (\Ord{\P} \cup -\Ord{\P})^\dual.
\end{equation}
Now, the convex hull of $\Ord{\P} \cup -\Ord{\P}$ is exactly the image of
$\TOrd{\dP_\circ}$ under the projection $\pi : \R^\P \times \R \rightarrow
\R^\P$ with $\pi(f,t) = \tfrac{1}{2}f$. Hence,
\[
    \ValP^\pm(\P) \ \cong \ \pi(\TOrd{\dP_\circ})^\dual \ \cong \
    \TOrd{\dP_\circ}^\dual \cap \mathrm{im}(\pi^*) \ \cong \ 
    \tprism{-2\ValP(\P)}\cap (\R^\P \times \{0\}),
\]
by Theorem~\ref{thm:twisted_val}.  If we now view $\tprism{-\ValP(\P)}$ as a
Cayley sum, we obtain
\begin{cor}\label{cor:val_pm}
    For any poset $\P$
    \[
        \ValP^\pm(\P)  \ = \ \ValP(\P) -
        \ValP(\P).
    \]
\end{cor}

A polytope $\Po$ with vertices in a lattice $\Lambda \subset \R^n$ is
\Defn{reflexive} if $\Po^\dual$ is a lattice polytope with respect to the dual
lattice $\Lambda^\vee := \{ \ell \in (\R^n)^* : \ell(x) \in \Z \text{ for all
} x \in \Lambda \}$. For two polytopes $\Po,\Qo \subset \R^n$, write
$\Gamma(\Po,\Qo):=\conv(\Po \cup -\Qo)$. Thus, $\Gamma(\Po,\Qo)$ is the
projection of $\CayDiff{\Po}{\Qo}$ onto the first $n$ coordinates.  The
polytopes $\Gamma(\Ord{\P},\Ord{\P})$ where studied by Hibi, Matsuda, and
Tsuchiya~\cite{Hibi15-1,Hibi15-2} in the context of Gorenstein polytopes,
i.e.\ lattice polytopes $\Po$ such that $r\Po$ is reflexive for some $r \in
\Z_{>0}$. By taking polars, we obtain the following from~\eqref{eqn:ValPpm}
and Corollary~\ref{cor:val_pm}.

\begin{cor}\label{cor:Gamma}
        For any poset $\P$,
        \[
            \Gamma(\Ord{\P},\Ord{\P}) \ = \ (\ValP(\P) -
            \ValP(\P))^\dual.
        \]
        In particular, $\Gamma(\Ord{\P},\Ord{\P})$ is reflexive.
\end{cor}

An explicit description of the face lattices of $\Val(\P)$, $\Val^\pm(\P)$ as
well as $\Gamma(\Ord{\P},\Ord{\P})$ can be obtained from
Theorem~\ref{thm:TO_faces}. 

This theorem also yields information about the polars of
$\TOrd{\P,\preceq_+,\preceq_-}$ for compatible double posets. For a poset that
is not compatible, the next result shows that the origin is not contained in
the interior of $\TOrd{\dP}$ and hence the polar is not bounded.

\begin{prop}\label{prop:0contained}
    Let $\dP$ be a double poset. Then $\TOrd{\dP}$ contains the origin
    in its interior if and only if $\dP$ is compatible.
\end{prop}
\begin{proof}
        If $\dP$ is compatible, then Corollary~\ref{cor:TO_ineq} shows that
    $\0$  strictly satisfies all facet-defining inequalities. If $\dP$ is
    not compatible, then it contains an alternating cycle $C$. It follows easily that $\ell_{C} \le 0$ on $\Ord{\P_+}$
    and $-\Ord{\P_-}$ and hence $\TOrd{\dP}$ is contained in the negative
    halfspace of $H = \{ (f,t) : \ell_{C}(f) \le 0\}$. Moreover, $\0 \in H
    \cap \TOrd{\dP}$, which shows that $\0 \not\in \relint \TOrd{\dP}$.
\end{proof}

\section{Anti-blocking polytopes}
\label{sec:AB}

\subsection{Anti-blocking polytopes and Minkowski sums}
\label{ssec:AB_Minkowski}
A polytope $\Po \subset \Rnn^n$ is called \Defn{anti-blocking} if
\begin{equation}\label{eqn:antiblock}
    q \in \Po \ \text{ and } \ 0 \le p \le q \quad \Longrightarrow \quad p \in
    \Po,
\end{equation}
where $p \le q$ refers to componentwise order in $\R^n$. The notion of
anti-blocking polyhedra was introduced by Fulkerson~\cite{Fulkerson} in
connection with min-max-relations in combinatorial optimization; our main
reference for anti-blocking polytopes is
Schrijver~\cite[Sect.~9.3]{Schrijver}. In this section, we consider the Cayley
sums 
\[
    \CayDiff{\Po}{\Qo} \ = \ \conv( \Po \times \{1\} \cup (-\Qo) \times \{-1\} ),
\]
where $\Po$ and $\Qo$ are anti-blocking polytopes. As before, we write
$\tprism{\Po}$ for $\CayDiff{\Po}{\Po}$. Our main source of examples will be the class of
stable set polytopes: For a graph $G = (V,E)$, a \Defn{stable set} is a subset
$S \subseteq V$ such that $\binom{S}{2} \cap E = \emptyset$. For simplicity,
we will assume that $V = [n]$ and we write $\1_S \in \{0,1\}^n$ for the
characteristic vector of a stable set $S$. The \Defn{stable set polytope} of
$G$ is the anti-blocking polytope
\[
    \Po_G \ := \ \conv( \1_S : S \subseteq V \text{ stable set} ) \ \subseteq \
    \R^n.
\]
The class of \emph{perfect} graphs is particularly interesting in this
respect. Lov\'{a}sz~\cite{Lovasz} characterized perfect graphs in terms of
their stable set polytopes and we use his characterization as a definition of
perfect graphs.  A \Defn{clique} of a graph $G = (V,E)$ is a subset $C
\subseteq V$ such that $\binom{C}{2} \subseteq E$. For a vector $x \in \R^n$
and a subset $J \subseteq [n]$, we write $x(J) = \sum_{j \in J} x_j$.

\begin{thm}[{\cite{Lovasz}}]\label{thm:perfect}
    A graph $G = ([n],E)$ is \Defn{perfect} if and only if
    \[
        \Po_G \ = \ \{ x \in \R^n : x \ge 0, x(C) \le 1 \text{ for all cliques
        } C \subseteq [n] \}.
    \]
\end{thm}

In this language, we can express the \emph{chain polytope} of a poset $\P$ as
a stable set polytope: The \Defn{comparability graph} $G(\P)$ of a poset
$(\P,\preceq)$ is the undirected graph with vertex set $\P$ and edge set $\{
xy : x \prec y \text{ or } y \prec x \}$. Note that cliques in $G(\P)$ are
exactly the chains of $\P$. For a poset $\P = ([n],\preceq)$ the comparability graph $G(\P)$ is perfect and hence
\[
    \Chain{\P} \ = \ \{ x \in \R^n : x \ge 0, x(C) \le 1 \text{ for all
    chains } C \subseteq [n] \} \ = \ \Po_{G(\P)}.
\]

\renewcommand\c{\mathbf{c}}%
\renewcommand\d{\mathbf{d}}%
If $\Po \subset \R^n$ is an anti-blocking polytope, then there are 
$\c_1,\dots,\c_r \in \Rnn^n$ such that
\newcommand\convDown[1]{\{#1\}^{\downarrow}}%
\newcommand\Vdown{V^\downarrow}%
\begin{equation}\label{eqn:downhull}
    \Po \ = \ \convDown{\c_1,\dots,\c_r} \ := \ \R^n_{\ge 0} \cap (
    \conv(\c_1,\dots,\c_r) - \R^n_{\ge 0}).
\end{equation}
The unique minimal such set, denoted by $\Vdown(P)$, is given by the minimal
elements of the vertex set of $\Po$ with respect to the partial order $\le$.
It also follows from~\eqref{eqn:antiblock} and the Minkowski--Weyl theorem
that there is a minimal collection $\d_1,\dots,\d_s \in \R^n_{\ge 0}$ such
that 
\[
    \Po \ = \ \{ \x \in \R^n : \x \ge 0, \inner{\d_i,\x} \le 1 \text{ for all }
    i=1,\dots,s\}
\]

For a polytope $\Qo \subseteq \Rnn^n$, its \Defn{associated} anti-blocking
polytope is the set
\newcommand\Ant[1]{A({#1})}%
\[
    \Ant{\Qo} \ := \ \{ \d \in \Rnn^n : \inner{\d,\x} \le 1 \text{ for all } \x
    \in \Qo \}.
\]
The following is the structure theorem for anti-blocking polytopes akin to the
bipolar theorem for convex bodies.

\begin{thm}[{\cite[Thm.~9.4]{Schrijver}}] \label{thm:doubleAnt}
    Let $\Po \subset \R^n$ be a full-dimensional anti-blocking polytope with
    \begin{align*}
        \Po \ = \ \convDown{\c_1,\dots,\c_r} &\ = \ \{ \x \in \R^n : \x \ge 0,
        \inner{\d_i,\x} \le 1 \text{ for all } i=1,\dots,s\} \\
    \intertext{for some $\c_1,\dots,\c_r,\d_1,\dots,\d_s \in \Rnn^n$.  Then}
        \Ant{\Po} \ = \ \convDown{\d_1,\dots,\d_s} &\ = \ \{ \x \in \R^n : \x
        \ge 0, \inner{\c_i,\x} \le 1 \text{ for all } i=1,\dots,r\}.
    \end{align*}
    In particular, $\Ant{\Ant{\Po}} = \Po$.
\end{thm}

\newcommand\ZeroOut[2]{{#1}^{[#2]}}%
Before we come to our first result regarding Cayley- and Minkowski-sums of
anti-blocking polytopes, we note the following fact. We write $V(\Po)$ for the
vertex set of a polytope $\Po$.

\begin{prop}\label{prop:ab_vert_union}
    Let $\Po_1, \Po_2$ be two full-dimensional anti-blocking polytopes. Then
    the vertices of $\conv(\Po_1 \cup -\Po_2)$ are exactly $(V(\Po_1) \cup
    V(-\Po_2)) \setminus \{\0\}$.
\end{prop}

For a polytope $\Po \subset \R^n$ and a vector $\c \in \R^n$, we denote by
$\Po^\c$ the face of $\Po$ that maximizes the linear function $\x \mapsto
\inner{\c,\x}$.

\begin{proof}
    It suffices to show that every $\v \in V(\Po_1) \setminus \{\0\}$ is a
    vertex of $\conv(\Po_1 \cup -\Po_2)$.  Let $\c \in \R^n$ such that $\Po_1^\c
    = \{\v\}$.  Since $\v \neq 0$, there is some $\d \in \R^n_{\ge 0}$ such
    that $\inner{\d, \u_1} \le 1$ for all $\u_1 \in \Po_1$ and $\inner{\d, \v}
    = 1$.  Hence, for any $\mu \ge 0$, $\Po_1^{\c + \mu \d} = \{v\}$. Now,
    $\inner{\d,-\u_2} \le 0$ for all $\u_2 \in P_2$. In particular, for $\mu >
    0$ sufficiently large, 
    \[
        \inner{\c + \mu \d,\u_2}  \ \le \ \inner{\c,\u_2}  \ < \ \mu +
        \inner{\c,\v} \ = \  \inner{\c + \mu \d,\v},
    \]
    which shows that $\v$ uniquely maximizes $\inner{\c + \mu\d,\u}$ over
    $\conv(\Po_1 \cup -\Po_2)$.
\end{proof}

For $\d \in \R^n_{\ge 0}$ and $I \subseteq [n]$, we write $\ZeroOut{\d}{I}$
for the vector with
\[
    (\ZeroOut{\d}{I})_j \ = \
        \begin{cases}
            d_j & \text{ for } j \in I \\
            0 & \text{ otherwise}.
        \end{cases}
\]

\begin{thm}\label{thm:CayleyAnti}
    Let $\Po_1,\Po_2 \subset \R^n$ be full-dimensional anti-blocking
    polytopes. Then
    \[
        (\Po_1 - \Po_2)^\dual  \ = \ \conv( \Ant{\Po_1} \cup -\Ant{\Po_2} ).
    \]
    Moreover, 
    \[
        (\CayDiff{2\Po_1}{2\Po_2})^\dual 
        \ = \ \CayDiff{-\Ant{\Po_2}}{-\Ant{\Po_1}}.
    \]
\end{thm}
\begin{proof}
    Let us denote the right-hand side of the first equation by $\Qo$. Note
    that $\inner{\u_1,-\v_2} \le 0$ for $\u_1 \in \Ant{P_1}$ and $\v_2 \in
    P_2$. This shows that $\inner{\u_1,\v} \le 1$ for all $\v \in P_1 - P_2$.
    By symmetry, this yields $\Qo \subseteq (\Po_1 - \Po_2)^\dual$.

    For the converse, observe that every vertex of $\Qo$ is of the form $\ZeroOut{\d}{I}$ with $\d\in\Vdown(\Ant{\Po_1})
    \cup -\Vdown(\Ant{\Po_2})$. It follows that $\z \in \Qo^\dual$ if and only if $
    \inner{\ZeroOut{\d}{I}, \z} \le 1$ for all $ \d \in \Vdown(\Ant{\Po_1})
    \cup -\Vdown(\Ant{\Po_2})$ and  all $I \subseteq [n]$.  For $\z \in
    \Qo^\dual$ write $\z = \z^1 - \z^2$ with $\z^1,\z^2 \ge 0$ and
    $\supp(\z^1) \cap \supp(\z^2) = \emptyset$, where for any $\z=(z_1,\dots,z_n)\in\R^n$ we set $\supp(\z):=\{i:z_i \neq 0\}$. We claim that $\z^i \in \Po_i$
    for $i=1,2$. Indeed, let $I = \supp(\z^1)$. Then for any $\d \in
    \Vdown(\Po_1)$ we have
    \[
        \inner{\d,\z^1} \ = \ \inner{\ZeroOut{\d}{I},\z} \ \le \ 1
    \]
    and hence $\z^1 \in \Po_1$. Applying the same argument to $\z^2$ shows
    that $\z \in \Po_1 - \Po_2$ and hence $(\Po_1 - \Po_2)^\dual \subseteq
    \Qo$.

    For the second claim, note that any linear function on $\R^n \times \R$
    that maximizes on a \emph{vertical} facet of $\CayDiff{2\Po_1}{2\Po_2}$ is
    of the form $\alpha_\d \inner{\d,\x} + \delta_\d t$ for $\d$ a vertex of
    $(\Po_1 - \Po_2)^\dual$ and some $\alpha_\d, \delta_\d \in \R$ with
    $\alpha_\d > 0$.  By the first claim and
    Proposition~\ref{prop:ab_vert_union}, it follows that $\d \in
    (V(\Ant{\Po_1}) \cup V(-\Ant{\Po_2})) \setminus \{\0\}$. 

    If $\d \in V(\Ant{\Po_1}) \setminus \{\0\}$, then $\inner{\d,\u_1} \le 1$
    is tight for $\u_1 \in \Po_1$ whereas $\inner{\d,-\u_2} \le 0$ is tight
    for $-\u_2 \in -\Po_2$.  Hence, 
    \[
        \inner{\d,\x} - t \ \le \ 1
    \]
    is the corresponding facet-defining halfspace. Similarly, if $-\d \in
    -V(\Ant{\Po_1}) \setminus \{\0\}$, then 
    \[
         \inner{-\d,x} + t \ \le \ 1
    \]
    is facet-defining. Together with the two horizontal facets $\inner{\0,\x}
    \pm t \le 1$ this yields an inequality description of
    $(\CayDiff{-\Ant{\Po_2}}{-\Ant{\Po_1}})^\dual$, which proves the second
    claim.
\end{proof}

Theorem~\ref{thm:CayleyAnti} together with Theorem~\ref{thm:doubleAnt} has a
nice implication that was used in~\cite{WSZ} in connection with Hansen
polytopes.

\begin{cor}\label{cor:AB_selfdual}
    For any full-dimensional anti-blocking polytope $\Po \subset \R^n$, the
    polytope $\CayDiff{\Po}{\Ant{\Po}}$ is linearly isomorphic to its polar 
    $(\CayDiff{\Po}{\Ant{\Po}})^\dual$. In particular, 
    $\CayDiff{\Po}{\Ant{\Po}}$ is self-dual.
\end{cor}

\subsection{Stable set polytopes of double graphs and double chain polytopes}
\label{ssec:double_graph}%
A \Defn{double graph} is a triple $\dG = (V,E_+,E_-)$ consisting of a node
set $V$ with two sets of edges $E_+,E_- \subseteq \binom{V}{2}$. Again, we
write $G_+ = (V, E_+)$ and $G_- = (V, E_-)$ to denote the two underlying
ordinary graphs.  The results of the preceding sections prompt the definition
of \Defn{stable set polytope} of a double graph
\[
    \Po_{\dG} \ := \ \CayDiff{2\Po_{G_+}}{2\Po_{G_-}}.
\]
For a double graph $\dG$, define the \Defn{complement graph} as
$\overline{\dG} = (V,E_-^c,E_+^c)$. Then Theorem~\ref{thm:CayleyAnti}
implies the following relation.

\begin{cor}\label{cor:dual_doublegraph}
    Let $\dG$ be a perfect double graph. Then $\Po_{\dG}^\dual$ is linearly
    isomorphic to $\Po_{\overline{\dG}}$.
\end{cor}
\begin{proof}
We have
\[
    \Po_{\dG}^\dual=(\CayDiff{2\Po_{G_+}}{2\Po_{G_-}})^\dual  \ = \
    \CayDiff{-\Ant{\Po_{G_-}}}{-\Ant{\Po_{G_+}}}  \ = \
    \CayDiff{-\Po_{\overline{G}_-}}{-\Po_{\overline{G}_+}} \ \cong \
    \Po_{\overline{\dG}}.
    \qedhere
\]
\end{proof}

In particular, a double poset $\dP = (\P,\preceq_\pm)$ gives rise to a double
graph $\dG(\dP) = (G(\P_+),G(\P_-))$ and the double chain polytope of $\dP$ is
simply $\TChain{\dP}  = \Po_{\dG(\dP)}$, the \Defn{double chain polytope} of
$\dP$. Theorem~\ref{thm:CayleyAnti} directly gives a facet description of the
double chain polytope. Note that compatibility is not required.

\begin{thm}\label{thm:TC_facets}
    Let $\dP$ be a double poset and $\TChain{\dP}$ its double chain polytope.
    Then $(g,t) \in \R^\P \times \R$ is contained in $\TChain{\dP}$ if and
    only if 
    \[
        \sum_{a \in C_+} g(a) - t \ \le \ 1 \quad \text{ and } \quad \sum_{a \in
        C_-} -g(a) + t \ \le \ 1,
    \]
    where $C_+ \subseteq \P_+$ and $C_- \subseteq \P_-$ ranges of all chains.
\end{thm}

For the usual order- and chain polytope, Hibi and Li~\cite{Hibi16} showed that $\Ord{\P}$ has at most as many facets as
$\Chain{\P}$ and equality holds if and only if $\P$ does not contain the $5$-element
poset with Hasse diagram 'X'. This is different in the case of double poset polytopes.

\begin{cor}
    Let $(\P,\preceq)$ be a poset. Then 
    $\TOrd{\dP_\circ}$ and $\TChain{\dP_\circ}$ have the same number of facets.
\end{cor}
\begin{proof}
    Alternating chains in $\dP_\circ$ are in bijection to twice the number of
    chains in $\P$.
\end{proof}

However, it is not true that $\TOrd{\dP_\circ}$ is always combinatorially isomorphic
to $\TChain{\dP_\circ}$.

\begin{example}\label{ex:XX}
    Let $\P$ be the $5$-element poset with Hasse diagram 'X'. Then the face
    vectors of $\TOrd{\dP_\circ}$ and $\TChain{\dP_\circ}$ are
    \begin{align*}
            f(\TOrd{\dP_\circ}) \ &= \ (16, 88, 204, 240, 144, 36) \\
            f(\TChain{\dP_\circ}) \ &= \ (16, 88, 222, 276, 162, 36).
    \end{align*}
\end{example}

Hibi and Li~\cite{Hibi16} conjectured that
$f(\Ord{\P})\le f(\Chain{\P})$ componentwise. Computations suggest that the
same relation should hold for the double poset polytopes of induced double
posets.

\begin{conj}\label{conj:f_dom}
    Let $\dP = (\P,\preceq,\preceq)$ be a double poset induced by a poset
    $(\P,\preceq)$. Then
    \[
        f_i(\TOrd{\dP}) \ \le \ f_i(\TChain{\dP})
    \]
    for $0 \le i \le |\P|$.
\end{conj}

An extension of the conjecture to general compatible double posets fails, as
the following example shows.

\begin{example}
    \newcommand\AltChain{\mathbf{A}}
    Let $\AltChain_n$ be an \Defn{alternating chain} of length $n$, that is,
    $\dP$ is a double poset on elements $a_1,a_2,\dots,a_{n+1}$ with cover
    relations
    \[
        a_1 \ \prec_+ \ a_2 \ \prec_- \ a_3 \ \prec_+ \cdots
    \]
    It follows from Theorem~\ref{thm:TC_facets} that the number of facets of
    $\TChain{\AltChain_n}$ is $3n+4$. Since $\AltChain_n$ is compatible, then
    by Theorem~\ref{thm:compat_facets} the number of facets of
    $\TOrd{\AltChain_n}$ equals the number of alternating chains which is
    easily computed to be $\binom{n+3}{2}+1$. Thus, for $n \ge 3$, the
    alternating chains $\AltChain_n$ fail Conjecture~\ref{conj:f_dom} for the
    number of facets.  For $n=3$, we explicitly compute
    \begin{align*}
            f(\TOrd{\AltChain_3}) \ &= \ ( 21, 70, 95, 60, 16) \quad \text{
            and} \\
            f(\TChain{\AltChain_3}) \ &= \ (21, 67, 86, 51, 13).
    \end{align*}
\end{example}

Every graph $G = (V,E)$ trivially gives rise to a double graph $\dG_\circ =
(V,E,E)$. Thus, the \Defn{Hansen polytope} of a graph $G$ is the polytope
$\Hansen{G} = \Po_{\dG_\circ}$. Theorem~\ref{thm:CayleyAnti} then yields a
strengthening of the main result of Hansen~\cite{Hansen}. Note that for the
complement graph $\overline{G} = (V,E^c)$, it follows that a subset $S
\subseteq V$ is a stable set of $G$ if and only if $S$ is a clique of
$\overline{G}$ and vice versa.  

\begin{cor}[{\cite[Thm.~4(c)]{Hansen}}]\label{cor:Hansen}
    Let $G$ be a perfect graph. Then $\Hansen{G}$ is $2$-level and
    $\Hansen{G}^\dual$ is affinely isomorphic to $\Hansen{\overline{G}}$.
\end{cor}
\begin{proof}
    By Theorem~\ref{thm:CayleyAnti} and Theorem~\ref{thm:perfect}
    \[
        \Hansen{G}^\dual \ = \ \CayDiff{-\Ant{\Po_G}}{-\Ant{\Po_G}} \ = \
        \CayDiff{-\Po_{\overline{G}}}{-\Po_{\overline{G}}} \ \cong \
        \Hansen{\overline{G}},
    \]
    which proves the second claim. A vertex of $\Hansen{G}^\dual$ is of the
    form $\d = \pm(-\1_C,1)$ for some clique
    $C$ of $G$. Thus, for any vertex $\v = \pm(2\1_S,1) \in \Hansen{G}$, where
    $S$ is a stable set of $G$, we compute $\inner{\d,\v} = \pm( 1 - 2|S \cap
    C|) = \pm 1$.
\end{proof}

\begin{example}[Double chain polytopes of dimension-two posets]
    Following Example~\ref{ex:2dim}, let $\pi_+,\pi_-\in\Z^n$ be two integer
    sequences with associated posets $\P_{\pi_+}$ and $\P_{\pi_-}$ of order
    dimension two. Consider the double posets $\dP=(\P_{\pi_+},\P_{\pi_-})$
    and $-\dP=(\P_{-\pi_-},\P_{-\pi_+})$.  We have
    \[
    \overline{\dG(\dP)} \ = \
    (\overline{G(\P_{\pi_-})},\overline{G(\P_{\pi_+})}) \ = \
    (G(\P_{-\pi_-}),G(\P_{-\pi_+})) \ = \ \dG(-\dP)
    \]
    and hence
    \[
    \TChain{\dP}^\dual \ \cong \ \TChain{-\dP}
    \]
    by Corollary \ref{cor:dual_doublegraph}. However, it is not necessarily
    true that $\TOrd{\dP}^\dual\cong\TOrd{-\dP}$, as can be checked for the
    double poset induced by the $X$-poset; cf.\ Example~\ref{ex:XX}.
\end{example}

\begin{example}[Double chain polytopes of plane posets]
Let $\dP$ be a plane double poset. By the last example, the double chain polytope $\TChain{\dP}$ is linearly equivalent to its polar $\TChain{\dP}^\dual$.
\end{example}

Among the $2$-level polytopes, independence polytopes of perfect graphs play a
distinguished role.  The following observation, due to Samuel Fiorini
(personal communication), characterizes $2$-level anti-blocking polytopes.
\begin{prop}\label{prop:Sam}
    Let $\Po$ be a full-dimensional anti-blocking polytope. Then $\Po$ is
    $2$-level if and only if $\Po$ is linearly isomorphic to $\Po_G$ for some
    perfect graph $G$.
\end{prop}
\begin{proof}
    The origin is a vertex of $\Po$ and, since $\Po$ is full-dimensional and
    anti-blocking, its neighbors are $ \alpha_1 \e_1, \dots, \alpha_1 \e_n$
    are for some $\alpha_i > 0$. After a linear transformation, we can assume
    that $\alpha_1 = \cdots = \alpha_n = 1$. Since $\Po$ is $2$-level, $\Po =
    \{ \x \in \Rnn^n : \inner{\d_i, \x} \le 1 \text{ for } i=1,\dots,s\}$
    where $\d_i \in \{0,1\}^n$ for all $i=1,\dots,s$. Let $G = ([n],E)$ be the
    minimal graph with cliques $\supp(\d_i)$ for all $i=1,\dots,s$. That is,
    $E = \bigcup_i \binom{\supp(\d_i)}{2}$. We have $\Po_G \subseteq \Po$.
    Conversely, any vertex of $\Po$ is of the form $\1_S$ for some $S
    \subseteq [n]$ and $\inner{\d_i,\1_S} = |\supp(\d_i) \cap S| \le 1$ shows
    that $\Po \subseteq \Po_G$. 
\end{proof}

This implies a characterization of the $2$-level polytopes among Cayley
sums of anti-blocking polytopes.

\begin{cor}\label{cor:AB_2l}
    Let $\Po_1,\Po_2 \subset \R^n$ be full-dimensional anti-blocking
    polytopes.  Then $\Po = \CayDiff{\Po_1}{\Po_2}$ is $2$-level if and only if
    $\Po$ is affinely isomorphic to $\Hansen{G}$ for some perfect graph $G$.
\end{cor}
\begin{proof}
    Sufficiency is Hansen's result (Corollary~\ref{cor:Hansen}). For
    necessity, observe that $\Po_1$ and $\Po_2$ are faces and hence have to be $2$-level. By the proof of Proposition~\ref{prop:Sam}, we may assume that $\Po_1 =  \Po_{G_1}$ for some perfect graph $G_1$ and $\Po_2 =  A\Po_{G_2}$ for some perfect $G_2$ and a diagonal matrix $A\in\R^{n\times n}$ with diagonal entries $a_i>0$ for $i\in[n]$. We will proceed in two steps: We first prove that $A$ must be the identity matrix and then show that $G_1=G_2$.
 
 For every $i\in[n]$ the inequality $x_i\ge 0$ is facet-defining for $\Po_1$. Hence it induces a facet-defining inequality for $\Po$, which must be of the form
  \[
\ell_i\ := \ -b_ix_i+t \ \le \ 1
 \]
 for some $b_i>0$, where $t$ denotes the last coordinate in $\R^{n+1}$. Observe that $\ell_i$ takes the values $1$ and $1-b_i$ on the vertices $\{\0,\e_i\}\times \{1\}$ of  the face $\Po_1\times\{1\}$. On the other hand, on $\{\0,-a_i\e_i\}\times \{-1\}\subset-\Po_2\times\{-1\}$, the values are $-1$ and $-1+a_ib_i$. Now $2$-levelness implies $a_i=1$ and $b_i=2$.
        
It now follows from Theorem~\ref{thm:CayleyAnti} that the facet-defining inequalities for $\Po$ are
\[
\begin{split}
2\1_{C_1}(\x)-t \ &\le \ 1\text{ and}\\
-2\1_{C_2}(\x)+t\ &\le \ 1,
\end{split}
\]        
where $C_1$ and $C_2$ are cliques in $G_1$ and $G_2$, respectively. By $2$-levelness each of these linear functions takes the values $-1$ and $1$ on the vertices of $\P$. This easily implies that every clique in $G_1$ must be a clique in $G_2$ and conversely. Hence $G_1=G_2$.
    \end{proof}

\subsection{Canonical Subdivisions}\label{ssec:AB_subdiv}
We now turn to the canonical subdivisions of $\Po_1 - \Po_2$ and
$\CayDiff{\Po_1}{\Po_2}$ for anti-blocking polytopes $\Po_1,\Po_2$.  A
\Defn{subdivision} of $\Po = \Po_1 - \Po_2$ is a collection of polytopes
$\Qo^1,\dots,\Qo^m \subseteq \Po$ each of dimension $\dim \Po$ such that $\Po
= \Qo^1\cup \cdots \cup \Qo^m$ and $\Qo^i \cap \Qo^j$ is a face of both for
all $i \neq j$. We call the subdivision \Defn{mixed} if each $\Qo^i$ is of the
form $\Qo^i_1 - \Qo^i_2$ where $\Qo^i_j$ is a vertex-induced subpolytope of
$\Po_j$ for $j=1,2$. Finally, a mixed subdivision is \Defn{exact} if $\dim
\Qo^i = \dim \Qo^i_1 + \dim \Qo^i_2$. That is, $\Qo^i$ is linearly isomorphic
to the Cartesian product $\Qo^i_1 \times \Qo^i_2$.  For a full-dimensional
anti-blocking polytope $\Po \subset \R^n$, every index set $J \subseteq [n]$
defines a distinct face $\Po|_J := \{ x \in \Po : x_j = 0 \text{ for } j
\not\in J \}$. This is an anti-blocking polytope of dimension $|J|$. For
disjoint $I,J \subseteq [n]$, the polytopes $\Po_1|_I, \Po_2|_J$ lie in
orthogonal subspaces and $\Po_1|_I -  \Po_2|_J$ is in fact a Cartesian
product. In this case, the Cayley sum $\CayDiff{\Po_1|_I}{\Po_2|_J}$ is 
called a \Defn{join} and denoted by $ \Po_1|_I * \Po_2|_J$. As with the
Cartesian product, the combinatorics of $ \Po_1|_I * \Po_2|_J$ is completely
determined by the combinatorics of $ \Po_1|_I$ and $\Po_2|_J$.

\begin{lem}\label{lem:AB_canonical}
    Let $\Po_1, \Po_2 \subset \R^n$ be full-dimensional anti-blocking
    polytopes.  Then $\Po_1 - \Po_2$ has a regular exact mixed subdivision
    with cells $\Po_1|_{J} - \Po_2|_{J^c}$ for all $J \subseteq [n]$. In
    particular, $\CayDiff{\Po_1}{\Po_2}$ has a regular subdivision into joins
    $\Po_1|_{J} * \Po_2|_{J^c}$ for all $J \subseteq [n]$.
\end{lem}

We call the subdivisions of Lemma~\ref{lem:AB_canonical} the \Defn{canonical
subdivisions} of $\Po_1 - \Po_2$ and $\CayDiff{\Po_1}{\Po_2}$, respectively.

\begin{proof}
    By the Cayley trick~\cite[Thm~9.2.18]{DLRS}, it is suffices to prove only
    the first claim. The subdivision of $\Po_1 - \Po_2$ is very easy to
    describe: Let us first note that the polytopes $\Po_1|_{J} - \Po_2|_{J^c}$
    for $J \subseteq [n]$ only meet in faces. Hence, we only need to verify that
    they cover $\Po_1 - \Po_2$. It suffices to show that for any point $\x \in
    \Po_1 - \Po_2$ with $x_i \neq 0$ for all $i$, there is a $J \subseteq [n]$
    with $\x \in \Po_1|_{J} - \Po_2|_{J^c}$. Let $\x_1,\x_2 \in
    \Rnn^n$ with $\x = \x_1 - \x_2$ and $\supp(\x_1) \cap \supp(\x_2) =
    \emptyset$.  We claim that $\x_i \in \Po_i$ for $i=1,2$. Indeed, if $\x =
    \y_1 - \y_2$ for some $\y_i \in \Po_i$, then $0 \le \x_i \le \y_i$ and
    $\x_i \in \Po_i$ by~\eqref{eqn:antiblock}.  In particular, 
    $\x_1 \in \Po_1 |_J$ and $\x_2 \in \Po_2 |_{J^c}$ and therefore $\x \in 
    \Po_1|_{J} - \Po_2|_{J^c}$.

    To show regularity, let $\omega : \R^n \times \R^n \rightarrow \R$ be the
    linear function such that $\omega(\e_i,0) = -\omega(0,\e_j) = 1$ for all
    $i,j = 1,\dots,n$. Then $\omega$ induces a mixed subdivision by picking
    for every point $\x \in \Po_1 - \Po_2$, the unique cell $F_1 - F_2$ such
    that
    $\x = \x_1 - \x_2$ with $\x_i \in \relint F_i$ and $(\x_1,\x_2)$
    minimizes $\omega$ over the set
    \[
        \{ (\y_1,\y_2) \in \Po_1 \times \Po_2 : \x = \y_1 - \y_2 \};
    \]
    see Section~9.2.2 of de Loera \emph{et al.}~\cite{DLRS} for more details.
    If $\omega$ is not generic, one has to be careful as the minimizer is not
    necessarily unique but in our case, we observe that for any $\y_i \in
    \Po_i$ with $\x = \y_1 - \y_2$ we have $\omega(\y_1,\y_2) >
    \omega(\x_1,\x_2)$ for all $(\y_i,\y_2)\neq(\x_1,\x_2)$ with $(\x_1,\x_2)$ defined above.
\end{proof}

\begin{figure}[h]
  \centering
    \includegraphics[width=0.6\textwidth]{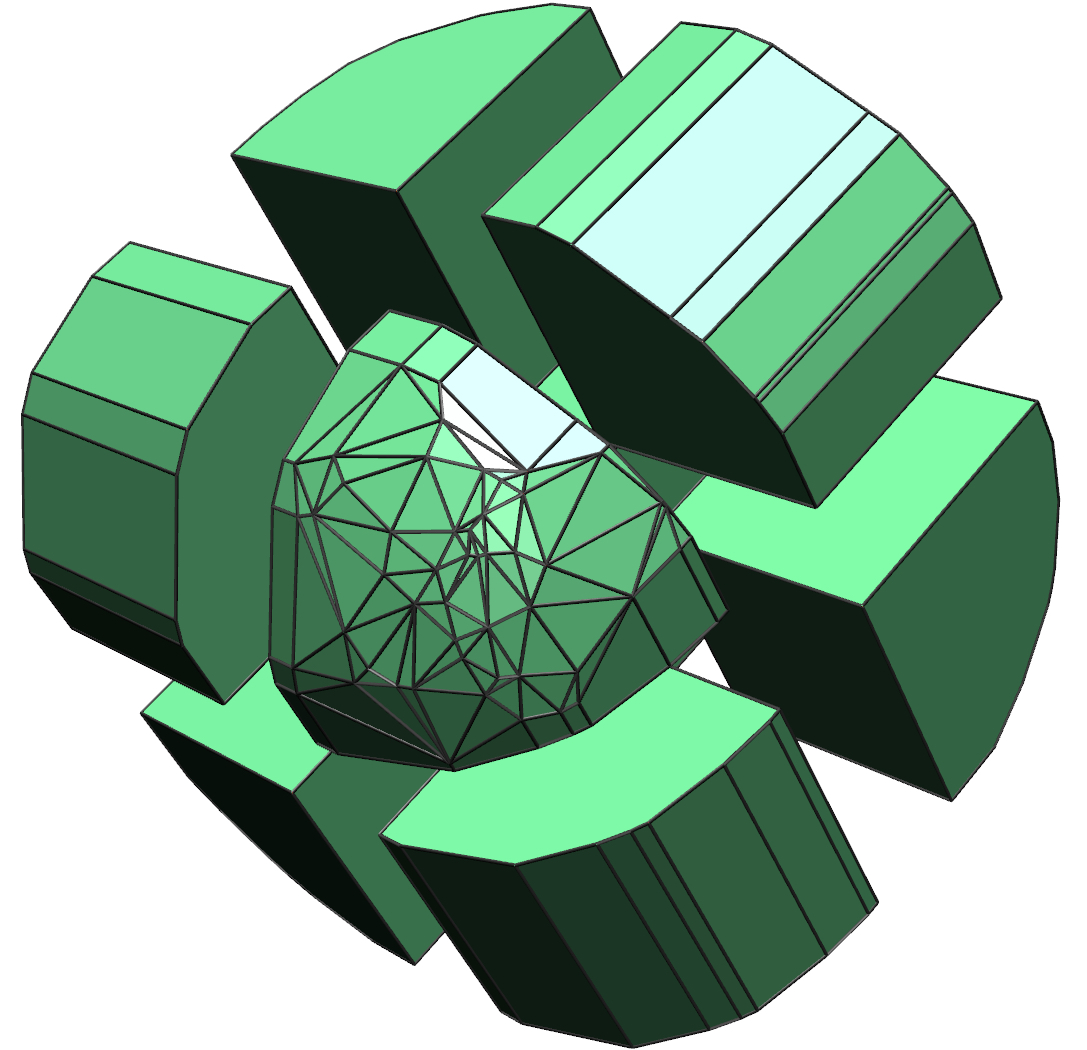}
    \caption{The canonical subdivision of $\Po_1 - \Po_2$ for two
    (random) anti-blocking polytopes $\Po_1,\Po_2 \subset \R^3_{\ge 0}$.  }
    \label{fig:can_subdiv}
\end{figure}

We define a \Defn{triangulation} of a polytope to be a subdivision into
simplices without new vertices. For a polytope with vertices in an affine
lattice $\AffLat$, a triangulation is \Defn{unimodular} if each simplex is
unimodular or, equivalently, has normalized volume $=1$.  A triangulation is
\Defn{flag} if any minimal non-face is of dimension $1$. This property implies
that the underlying simplicial complex is completely determined by its graph.

\begin{thm}\label{thm:AB_subdiv}
    Let $\Po_1,\Po_2 \subset \R^n$ be full-dimensional anti-blocking
    polytopes with subdivisions $\subdiv_1$ and $\subdiv_2$, respectively. For
    $J \subseteq [n]$, let $\subdiv_i|_J := \{ S \cap \Po_i|_J : S \in
    \subdiv_i \}$ be the restriction of $\subdiv_i$ to $\Po_i|_J$ for $i=1,2$.
    Then
    \[
        \subdiv \ := \ \bigcup_{J \subseteq [n]} \subdiv_1|_J * \subdiv_2|_{J^c}
    \]
    is a subdivision of $\CayDiff{\Po_1}{\Po_2}$. In particular,
    \begin{compactenum}[\rm (i)]
        \item If $\subdiv_1$ and $\subdiv_2$ are regular, then $\subdiv$ is
            regular.
        \item If $\subdiv_1$ and $\subdiv_2$  are unimodular
            triangulations with respect to $\Lambda$,
            then $\subdiv$ is a unimodular triangulation with respect to the
            affine lattice $\Lambda \times (2\Z + 1)$. 
        \item If $\subdiv_1$ and $\subdiv_2$ are flag, then $\subdiv$ is flag.
    \end{compactenum}
\end{thm}
Note that {\rm (iii)} also holds if the triangulations use more lattice points
than just the vertices.
\begin{proof}
    For the first claim, observe that $\subdiv_i|_J$ is a subdivision of the
    face $\Po_i|_J$. By~\cite[Thm~4.2.7]{DLRS}, $\subdiv_1|_J *
    \subdiv_2|_{J^c}$ is a subdivision of $\Po_1|_J * \Po_2|_{J^c}$. Hence,
    $\subdiv$ is a refinement of the canonical subdivision of
    Lemma~\ref{lem:AB_canonical}.

    If $\subdiv_i$ is a regular subdivision of $\Po_i$, then there are weights
    $\omega_i : V(\Po_i) \rightarrow \R$ for $i=1,2$. By adding a constant
    weight to every vertex if necessary, we can assume that $\omega_1(\v_1) >
    0$ and $\omega_2(\v_2) < 0$ for all $\v_1 \in V(\Po_1)$ and $\v_2 \in
    V(\Po_2)$. Again using the Cayley trick, it is easily seen that $\omega :
    V(\CayDiff{\Po_1}{\Po_2}) \rightarrow \R$ given by $\omega(\v_1,+1) :=
    \omega_1(\v_1)$ and $\omega(\v_2,-1) := \omega_2(\v_2)$ induces $\subdiv$.
    
    Claim (ii) simply follows from the fact that the join of two unimodular
    simplices is unimodular. 
    
    For claim (iii), let $\sigma = \sigma_1 \uplus \sigma_2 \subseteq
    V(\CayDiff{\Po_1}{\Po_2})$ be a minimal non-face. Since $\subdiv_1$ and
    $\subdiv_2$ are flag, it follows that $\sigma_1 \in \subdiv_1$ and
    $\sigma_2 \in \subdiv_2$. Thus, there vertices $v_i \in \sigma_i$ for
    $i=1,2$ such that $\supp(v_1) \cap \supp(v_2) \neq \emptyset$ but $\sigma
    \setminus \{v_i\}$ is a face for $i=1$ and $i=2$. But then $\{v_1,v_2\}$
    is already a non-face and the claim follows.
\end{proof}

The theorem has some immediate consequences.

\begin{cor}
    Let $\Po_1, \Po_2 \subset \R^n$ be two full-dimensional anti-blocking
    polytopes with vertices in a given lattice. If $\Po_1, \Po_2$ have
    unimodular triangulations, then $\Po_1 - \Po_2$ and $\Gamma(\Po_1,\Po_2) =
    \conv(\Po_1 \cup -\Po_2)$ also have unimodular triangulations.
\end{cor}
\begin{proof}
    By Theorem~\ref{thm:AB_subdiv} and the Cayley trick, $\Po_1 - \Po_2$ has a
    mixed subdivision into Cartesian products of unimodular simplices.
    Products of unimodular simplices are $2$-level and, for example 
    by~\cite[Thm.~2.4]{Sullivant}, have unimodular triangulations. The
    polytope $\conv(\Po_1 \cup -\Po_2)$ inherits a triangulation from the
    upper or lower hull of $\CayDiff{\Po_1}{\Po_2}$, which has a unimodular
    triangulation by Theorem~\ref{thm:AB_subdiv}.
\end{proof}

\begin{cor}
    Let $\dG$ be a perfect double graph. Then $\Po_{\dG}$, $\Po_{G_+} -
    \Po_{G_-}$, and $\Gamma(\Po_{G_+},\Po_{G_-})$ have regular
    unimodular triangulations.
\end{cor}
\begin{proof}
    By Theorem~\ref{thm:perfect}, both polytopes $\Po_{G_+}$ and $\Po_{G_-}$
    are $2$-level and by~\cite[Thm.~2.4]{Sullivant} have
    unimodular triangulations. The result now follows from the previous
    corollary.
\end{proof}

\subsection{Lattice points and volume} \label{ssec:AB_ehrhart}

Lemma~\ref{lem:AB_canonical} directly implies a formula for the (normalized)
volume of $\CayDiff{\Po_1}{\Po_2}$ in terms of the volumes of the
anti-blocking polytopes $\Po_1,\Po_2$.

\begin{cor}\label{cor:AB_volumes}
    Let $\Po_1,\Po_2 \subset \R^n$ be full-dimensional anti-blocking
    polytopes.  Then
    \[
        \vol(\Po_1 - \Po_2) \ = \  \sum_{J \subseteq [n]} \vol(\Po_1|_J)
        \vol(\Po_2|_{J^c}).
    \]
    If $\Po_1$ and $\Po_2$ have unimodular triangulations with respect to a
    lattice $\Lambda$, then the normalized volume of $\CayDiff{\Po_1}{\Po_2}$
    with respect to the affine lattice $\Lambda \times (2\Z+1)$ is
    \[
        \nvol(\CayDiff{\Po_1}{\Po_2}) \ = \ 
        \sum_{J \subseteq [n]} 
        \nvol(\Po_1|_J)
        \nvol(\Po_2|_{J^c}).
    \]
\end{cor}
\begin{proof}
    Both claims follow from Lemma~\ref{lem:AB_canonical}.
    For the second statement, note that Theorem~\ref{thm:AB_subdiv} yields
    that $\CayDiff{\Po_1}{\Po_2}$ has a unimodular triangulation and hence its
    normalized volume is the number of simplices of maximal dimension, which
    is the number in the right-hand side.
\end{proof}

If $\Po_1,\Po_2 \subset \R^n$ are \emph{rational} anti-blocking polytopes,
then so are $\CayDiff{2\Po_1}{2\Po_2}$ and $\Po_1 - \Po_2$. Our next goal is
to determine their Ehrhart quasi-polynomials for a particular interesting
case. We briefly recall the basics of Ehrhart theory; for more see, for
example,~\cite{BR,crt}. If $\Po \subset \R^n$ is a $d$-dimensional polytope
with rational vertex coordinates, then the function $\ehr_\Po(k) := |k\Po \cap
\Z^n|$ agrees with a quasi-polynomial of degree $d$. We will identify
$\ehr_\Po(k)$ with this quasi-polynomial, called the \Defn{Ehrhart
quasi-polynomial}. If $\Po$ has its vertices in $\Z^n$, then $\ehr_\Po(k)$ is
a polynomial of degree $d$.  If $\Po$ is full-dimensional, then the leading
coefficient of $\ehr_\Po(k)$ is $\vol(\Po)$. We will need the following
fundamental result of Ehrhart theory.

\begin{thm}[Ehrhart--Macdonald theorem]\label{thm:EM}
    Let $\Po \subset \R^n$ be a rational polytope of dimension $d$, then
    \[
        (-1)^d \ehr_{\Po}(-k) \ = \ | \relint(k\Po) \cap \Z^n|.
    \]
\end{thm}

We call an anti-blocking polytope $\Po \subset \R^n$ \Defn{dual integral} if
$A(\Po)$ has all vertices in $\Z^n$. By Theorem~\ref{thm:doubleAnt}, this means
that there are $\d_1,\dots,\d_s \in \Znn^n$ such that
\[
    \Po \ = \ \{ \x \in \R^n : \x \ge 0, \inner{\d_i,\x} \le 1 \text{ for }
    i=1,\dots,s \}.
\]

\begin{cor}\label{cor:AB_ehrhart}
    Let $\Po_1,\Po_2 \subset \R^n$ be two full-dimensional rational
    anti-blocking polytopes. If $\Po_1$ is dual integral, then
    \[
        \ehr_{\Po_1 - \Po_2}(k) \ = \ \sum_{J \subseteq [n]}
        (-1)^{|J|}\ehr_{\Po_1|_J}(-k-1) \ehr_{\Po_2|_{J^c}}(k).
    \]
\end{cor}

The Corollary is simply deduced from Theorem~\ref{thm:EM} and the following
stronger assertion. For a set $S \subset \R^n$, let us write $E(S) := |S \cap
\Z^n|$.

\begin{thm}\label{thm:AB_dvol}
    Let $\Po_1,\Po_2 \subset \R^n$ be two full-dimensional rational
    anti-blocking polytopes and assume that $\Po_1$ is dual integral. For any
    $a,b \in \Z_{>0}$
    \[
        E(a\Po_1 - b\Po_2) \ = \ 
        | (a\Po_1 - b\Po_2) \cap \Z^n | \ = \
        \sum_{J \subseteq [n]} E(\relint((a+1)\Po_1))
        \cdot E(b\Po_2).
    \]
\end{thm}
\begin{proof}
    It follows from Lemma~\ref{lem:AB_canonical} that for any $a,b \in
    \Z_{>0}$, 
    \[
        a\Po_1 - b\Po_2 \ = \ \bigcup_{J \subseteq [n]} (a\Po_1|_J -
        b\Po_2|_{J^c}).
    \]
    For $J \subseteq [n]$, the cell $ a\Po_1|_J - b\Po_2|_{J^c}$ is contained
    in the orthant $\Rnn^J \times \R_{\le 0}^{J^c}$. It is easy to see that 
    \[ 
        \Z^n \ = \ \biguplus_{J \subseteq [n]} \Z^J_{> 0} \times \Z^{J^c}_{\le
        0} 
    \] 
    is a partition and for each $J \subseteq [n]$
    \[
        (a\Po_1 - b\Po_2) \cap (\Z^J_{> 0} \times \Z^{J^c}_{\le 0})
        \ = \ 
        (a\Po_1|_J - b\Po_2|_{J^c}) \cap (\Z^J_{> 0} \times \Z^{J^c}_{\le 0})
        \ = \ (a\Po_1|_J \cap \Z_{>0}^J) - (b\Po_2|_{J^c} \cap \Z^{J^c}).
    \]
    If $\Po_1$ is dual integral, then $\Po_1|_J$
    is dual integral. Thus, for a fixed $J$, there are $\d_1,\dots,\d_s \in
    \Z_{\ge 0}^{J}$ such that 
    \begin{align*}
        (a\Po_1|_J \cap \Z^J_{>0})
        &\ = \ \{ \x \in \Z^J : \x > 0, \inner{\d_i,\x} \le a \}\\
        &\ = \ \{ \x \in \Z^J : \x > 0, \inner{\d_i,\x} < a+1 \}
        \ = \ \relint((a+1) \Po_1|_J) \cap \Z^J.
    \end{align*}
    This proves the result.
\end{proof}

Clearly, it would be desirable to apply Corollary~\ref{cor:AB_ehrhart} to the
case that $\Po_1$ is a lattice polytope as well as dual integral. 

\begin{prop}\label{AB_perfect}
    Let $\Po \subset \R^n$ be a full-dimensional dual-integral anti-blocking
    polytope with vertices in $\Z^n$. Then $\Po = \Po_G$ for some perfect
    graph $G$.
\end{prop}
\begin{proof}
    Let $\Po$ be given by 
    \[
        \Po \ = \ \{ \x \in \R^n : \x \ge 0, \inner{\d_i,\x} \le 1 \text{ for
        } i=1,\dots,s \}
    \]
    for some $\d_1,\dots,\d_s \in \Znn^n$. Since $\Po$ is full-dimensional and
    a lattice polytope, it follows that $\e_1,\dots,\e_n \in \Po$ and for any
    $1 \le j \le s$ we compute
    \[
        0 \ \le \ \inner{\d_j, \e_i} \ \le \ 1
    \]
    for all $i$ and since the $\d_j$ are integer vectors, it follows that
    $\d_j = \1_{C_j}$ for some $C_j \subset [n]$. Consequently, the vertices
    of $\Po$ are in $\{0,1\}^n$ and $\Po$ is $2$-level.  By
    Proposition~\ref{prop:Sam}, $\Po = \Po_G$ for some perfect graph $G$. 
\end{proof}

This severely limits the applicability of Corollary~\ref{cor:AB_ehrhart}
to \emph{lattice} anti-blocking polytopes. On the other hand, we do not know of
many results regarding the Ehrhart polynomials or even volumes of stable set
polytopes of perfect graphs; see also the next section.

\begin{thm}\label{thm:Ehr_AB_Cay}
    Let  $\Po_1,\Po_2 \subset \R^n$ be two full-dimensional rational
    anti-blocking polytopes such that $\Po_1$ is dual integral. Then for $\Po
    := \CayDiff{2\Po_1}{2\Po_2}$
    \[
        \ehr_\Po(k) \ = \ 
        | k \Po \cap \Z^{n+1}| \ = \
        \sum_{J \subseteq [n]} (-1)^{|J|}\sum_{s=-k}^k
        \ehr_{\Po_1|_J}( s-k-1 ) \cdot \ehr_{\Po_2|_{J^c}}( k+s).
    \]
\end{thm}
\begin{proof}
    For $k > 0$, 
    \[
        k\Po \ = \ \conv( 2k\Po_1 \times \{k\} \cup -2k\Po_2 \times \{-k\} ).
    \]
    In particular, if $(\p,t)$ is a lattice point in $k\Po$, then $-k \le t
    \le k$. For fixed $t$, 
    \[
        \{ \p \in \Z^n : (\p,t) \in k\Po\} \ = \ \left((k-t)\Po_1 - (k+t)\Po_2
        \right) \cap \Z^n.
    \]
    Theorems~\ref{thm:AB_dvol} and~\ref{thm:EM} then complete the proof.
\end{proof}

\section{Triangulations and transfers}
\label{sec:triang}

If $\dP = \dP_\circ =  (\P,\preceq,\preceq)$ is induced by a single poset, then
Corollaries~\ref{cor:TO-2l} and~\ref{cor:AB_2l} assure us that
$\TOrd{\dP_\circ}$ and $\TChain{\dP_\circ}$ are 2-level
and~\cite[Thm.~2.4]{Sullivant} implies that both polytopes have unimodular
triangulations with respect to the affine lattice $\AffLat = 2\Z^\P \times
(2\Z + 1)$.  In this section we give explicit triangulations of the double
chain polytope $\TChain{\dP}$ and, in the compatible case, of the double order
polytope $\TOrd{\dP}$.  These triangulations will be \emph{regular},
\emph{unimodular}, and \emph{flag}.  To that end, we will generalize Stanley's
approach~\cite{TwoPoset} from poset polytopes to double poset polytopes.  We
put the triangulation to good use and explicitly compute the Ehrhart
polynomial and the volume of $\TChain{\dP}$ and, in case that $\dP$ is
compatible, of $\TOrd{\dP}$.

\subsection{Triangulations of double poset polytopes}
\label{ssec:TO_triang}

Recall from the introduction that for a poset $(\P,\preceq)$, the order
polytope $\Ord{\P}$ parametrizes all order preserving maps $f : \P \rightarrow
[0,1]$. Any $f \in \Ord{\P}$ induces a partial order $\P_f =
(\P,\preceq_f)$ by $a \prec_f b$ if $a \prec b$ or, when $a,b$ are
incomparable, if $f(a) < f(b)$. Clearly, $\preceq_f$ refines $\preceq$ and
hence $\Ord{\P_f} \subseteq \Ord{\P}$. Since filters in $\P_f$ are filters in
$\P$, $\Ord{\P_f}$ is a vertex-induced subpolytope of $\Ord{\P}$. If $f$ is
\emph{generic}, that is, $f(a) \neq f(b)$ for all $a \neq b$, then $\preceq_f$
is a total order and $\Ord{\P_f}$ is a unimodular simplex of dimension $|\P|$.
Stanley showed that the collection of all these simplices constitute a
unimodular triangulation of $\Ord{\P}$. More precisely, this canonical
triangulation of $\Ord{\P}$ is given by the \Defn{order complex}
$\Delta(\Birk{\P})$ of $\Birk{\P}$, i.e., the collection of chains in the
Birkhoff lattice of $\P$ ordered by inclusion.  Since a collection of filters
$\Filter_0,\dots,\Filter_k$ is \emph{not} a chain if and only if $\Filter_i
\not\subseteq \Filter_j$ and $\Filter_j \not\subseteq \Filter_i$ for some $0
\le i,j \le k$, the canonical triangulation is \emph{flag}.

Stanley~\cite{TwoPoset} elegantly \emph{transferred} the canonical
triangulation of $\Ord{\P}$ to $\Chain{\P}$ in the following sense. Define the
\Defn{transfer map} $\Transfer_\P : \Ord{\P} \rightarrow \Chain{\P}$ by
\begin{equation}\label{eqn:transfer}
    (\Transfer_\P f)(b) \ := \ \min\{ f(b) - f(a) : a \prec b \},
\end{equation}
for $f \in \Ord{\P}$ and $b \in \P$. This is a \Defn{piecewise linear} map and
the domains of linearity are exactly the full-dimensional simplices
$\Ord{\P_f}$ for generic $f$. In particular, $\Transfer_\P(\1_\Filter) =
\1_{\min(\Filter)}$ for any filter $\Filter \subseteq \P$, which shows that
$\Transfer_\P$ maps $\Ord{\P}$ into $\Chain{\P}$. To show that $\Transfer_\P$
is a PL homeomorphism of the two polytopes, Stanley gives an explicit inverse
$\iTransfer_\P : \Chain{\P} \rightarrow \Ord{\P}$ by
\begin{equation}\label{eqn:itransfer}
    (\iTransfer_\P g)(b) \ := \ \max\{ g(a_0) + \cdots + g(a_{k-1}) +
    g(a_k) : a_0 \prec \cdots \prec a_{k-1} \prec a_k \preceq b \},
\end{equation}
for any $g \in \Chain{\P}$.
Note that our definition of $\iTransfer_\P$ differs from that in~\cite{TwoPoset} in
that we do \emph{not} require that the chain has to end in $b$. This will be
important later. It can be easily checked that  $\iTransfer_\P$ is an inverse to
$\Transfer_\P$. Hence, the simplices
\[
    \conv( \1_{\min(\Filter_0)}, \dots, \1_{\min(\Filter_k)} ) \quad \text{
    for } \quad \{ \Filter_0 \subseteq  \cdots \subseteq \Filter_k \} \in
    \Delta(\Birk{\P})
\]
constitute a flag triangulation of $\Chain{\P}$.

We will follow the same approach as Stanley but, curiously, it will be simpler
to start with a triangulation of $\TChain{\dP}$. Recall from
Section~\ref{ssec:TO_faces} that a pair of chains $C = C_+ \uplus C_-$ with
$C_\sigma \subseteq \Birk{\P_\sigma}$ is non-interfering if $\min(\Filter_+)
\cap \min(\Filter_-) = \emptyset$ for any $\Filter_\sigma \in C_\sigma$ for
$\sigma = \pm$.

\begin{cor}\label{cor:TC_triang}
    Let $\dP = (\P,\preceq_+,\preceq_-)$ be a double poset. Then a regular
    triangulation of $\TChain{\dP}$ is given as follows: The
    $(k-1)$-dimensional simplices are in bijection to non-interfering pairs of
    chains $C = C_+ \uplus C_+ \subseteq \TBirk{\dP}$ with $|C| = |C_+| +
    |C_-| = k$. Moreover, the triangulation is regular, unimodular (with
    respect to $\AffLat$), and flag.
\end{cor}
\begin{proof}
    The canonical triangulation of $\Chain{\P_\sigma}$ is regular, unimodular,
    and flag for $\sigma=\pm$. As described above, its
    $(l_\sigma-1)$-simplices are in bijection to chains $C_\sigma \subseteq
    \Birk{\P_\sigma}$ of length $|C_\sigma| = l_\sigma$. More
    precisely, the simplex corresponding to $C_\sigma$ is given by
    \[
        F(C_\sigma) = \conv( \1_{\min(\Filter_\sigma)} : \Filter_\sigma \in
        C_\sigma ).
    \]
    By Theorem~\ref{thm:AB_subdiv} applied to $\TChain{\dP} =
    \CayDiff{2\Chain{\P_+}}{2\Chain{\P_-}}$, it follows that a unimodular and
    flag triangulation is given by the joins $2F(C_+) * 2F(C_-)$ for all
    chains $C_\sigma \subseteq \Birk{\P_\sigma}$ such that $F(C_+)$ and
    $F(C_-)$ lie in complementary coordinate subspaces. This, however, is
    exactly the case when $\min(\Filter_+) \cap \min(\Filter_-) = \emptyset$
    for all $\Filter_\sigma \in C_\sigma$ for $\sigma=\pm$.
\end{proof}

Corollary~\ref{cor:TC_triang} gives a canonical triangulation that
combinatorially can be described as a subcomplex of $\Delta(\TBirk{\dP}) =
\Delta(\Birk{\P_+}) * \Delta(\Birk{\P_-})$, called the \Defn{non-interfering
complex}
\newcommand\niDelta[1]{\Delta^\mathrm{ni}(#1)}%
\[
    \niDelta{\dP} \ := \ \{ C : C = C_+ \uplus C_- \in
    \Delta(\TBirk{\dP}) ,\, C \text{ non-interfering} \}.
\]
Associating $\Delta(\Birk{\P})$ to a poset $\P$ is very natural and can be
motivated, for example, from an algebraic-combinatorial approach to the order
polynomial (cf.~\cite{crt}). It would be very interesting to know if the
association $\dP$ to $\niDelta{\dP}$ is equally natural from a purely
combinatorial perspective.

Given a double poset $\dP = (\P,\preceq_+,\preceq_-)$, we define a piecewise linear map $\TPsi{\dP} : \R^\P \rightarrow \R^\P$ by
\begin{equation}\label{eqn:TPsi}
    \TPsi{\dP}(g) \ := \ \iTransfer_{\P_+}(g) \ - \ \iTransfer_{\P_-}(-g),
\end{equation}
for any $g \in \R^\P$. Here, we use that $\iTransfer_{}$, as given 
in~\eqref{eqn:itransfer}, is defined on all of $\R^\P$ with
the following important property: For $g \in \R^\P$, let us write
$g = g^+ - g^-$, where $g^+,g^- \in \Rnn^\P$ with disjoint
supports. Then $\iTransfer_{\P_\sigma}(g) = \iTransfer_{\P_\sigma}(g^+)$
for $\sigma = \pm$. Thus, 
\[
    \TPsi{\dP}(g) \ = \ \iTransfer_{\P_+}(g^+) - \iTransfer_{\P_-}(g^-),
\]
for any $g \in \R^\P$.  In particular, $\TPsi{\dP}$ takes $\lambda
\Chain{\P_+} - \mu \Chain{\P_-}$ into $\lambda \Ord{\P_+} - \mu \Ord{\P_-}$
for any $\lambda,\mu \ge 0$.  Indeed, for any pair of antichains $A_\sigma
\subseteq \P_\sigma$, first observe that $\1_{A_+} - \1_{A_-} =
\1_{A_+\setminus A_-} - \1_{A_-\setminus A_+}$. Hence, it suffices to assume
that $A_+ \cap A_- = \emptyset$. We compute
\[
    \TPsi{\dP}(\1_{A_+} - \1_{A_-} ) \ = \ \1_{\Filter_+} -
    \1_{\Filter_-},
\]
where for $\sigma = \pm$, $\Filter_\sigma \subseteq \P_\sigma$ is the filter generated by
$A_\sigma$. If $\dP$ is a compatible double poset, then
Corollary~\ref{cor:TO_vertical_edges} implies that $\TPsi{\dP}$ is a
surjection on vertex sets.

\begin{lem}\label{lem:TPsi_iso}
    Let $\dP=(\Po,\preceq_+,\preceq_-)$ be a compatible double poset. Then $\TPsi{\dP} : \R^\P
    \rightarrow \R^\P$ is a lattice-preserving piecewise linear isomorphism. 
\end{lem}    
\begin{proof}
    It follows directly from \eqref{eqn:TPsi} that $\TPsi{\dP}$ is piecewise
    linear. To show that $\TPsi{\dP}$ is an isomorphism, we explicitly
    construct for $f\in\R^\P$ a $g \in \R^\P$ such that $\TPsi{\dP}(g)=f$. 
    Since $\dP$ is compatible, we can assume that $\P = \{a_1,\dots,a_n\}$
    such that $a_i \prec_+ a_j$ or $a_i \prec_- a_j$ implies $i < j$. 

    It follows from~\eqref{eqn:TPsi} that $\TPsi{\dP}(g')(a_1) = g'(a_1)$ for
    any $g' \in \R^\P$ and hence, we can set $g(a_1) := f(a_1)$. Now assume
    that $g$ is already defined on $D_k := \{ a_1,\dots, a_k \}$ for some $k$.
    For $g'\in\R^\P$ observe that 
    \[
        \iTransfer_{\P_+}(g')(a_{k+1}) \ = \ \max(g'(a_{k+1}),0) + r
    \]
    where $r=0$ or $r = \iTransfer_{\P_+}(g')(a_i)$ for some $i \le k$.
    Analogously,
    \[
        \iTransfer_{\P_-}(-g')(a_{k+1}) \ = \ \max(-g'(a_{k+1}),0) + s
    \]
    where $s=0$ or $s = \iTransfer_{\P_-}(-g')(a_j)$ for some $j \le k$.
    Thus, we set
    \[
        g(a_{k+1})\ := \  f(a_{k+1}) - r + s
    \]
    This uniquely determines $g$ by induction on $k$.  To prove that
    $\TPsi{\dP}$ is lattice-preserving, observe that by~\eqref{eqn:TPsi} we
    have $\TPsi{\dP}(\Z^\P)\subseteq\Z^\P$. Moreover, if $f=\TPsi{\dP}(g)$
    with $f\in\Z^\P$ and the above construction shows that $g\in\Z^\P$. Hence,
    $\TPsi{\dP}(\Z^\P)\subseteq\Z^\P$, which finishes the proof. 
\end{proof}

Using the notation from~\eqref{eqn:Tfaces} in Section~\ref{ssec:TO_faces}, the
lemma shows that 
\begin{equation}\label{eqn:TOrd_triang}
    \{ \Tface{C} : C \in \ \niDelta{\dP} \}
\end{equation}
is a realization of the flag simplicial complex $\niDelta{\dP}$ by
unimodular simplices inside $\TOrd{\dP}$. Using Gr\"obner bases in
Section~\ref{sec:GB}, we will show the following result.

\begin{thm}\label{thm:T_iso}
    Let $\dP = (\P,\preceq_+,\preceq_-)$ be a compatible double poset.
    Then the map 
    \[
        (g,t) \mapsto (\TPsi{\dP}(g),t)
    \]
    is a piecewise linear homeomorphism from $\R^{\P}
    \times \R$ to itself that preserves the lattice $\Z^\P \times \Z$.
    In particular, it maps $\TChain{\dP}$ to $\TOrd{\dP}$ and hence
    $\niDelta{\dP}$ is a regular, unimodular, and flag triangulation of
    $\TOrd{\dP}$.
\end{thm}
\begin{proof}
    By the previous lemma, \eqref{eqn:TOrd_triang} is a realization of
    $\niDelta{\dP}$ in $\TOrd{\dP}$ without new vertices. Moreover, every
    maximal simplex contains the edge $e = \conv\{(\0,1), (\0,-1)\}$. Hence,
    it suffices to show that for every maximal simplex in $\niDelta{\dP}$, the
    supporting hyperplane of every facet not containing $e$ is supporting for
    $\TOrd{\dP}$.

    Let $C = \{ \Filter_{+0} \subset \cdots \subset \Filter_{+k}, \Filter_{-0}
    \subset \cdots \subset \Filter_{-l}\}$ be two maximal non-interfering
    chains. Set $A_{+i} := \min(\Filter_{+i})$ for $1 \le i \le k$ and
    $A_{-0},\dots,A_{-l}$ likewise. It follows that $\P_1 = \bigcup A_{+i}$
    and $\P_2 = \bigcup A_{-j}$ give a partition of $\P$. In particular, since
    $C$ was maximal, we have that $\{a^+_{k-i-1}\} = A_{+i} \setminus
    A_{+(i-1)}$ and $\P_1 = \{a^+_1,\dots,a^+_k\}$. In particular, if $a^+_s
    \prec_+ a^+_t$, then $s < t$. The same argument yields $\P_2 =
    \{a^-_1,\dots,a^-_l\}$ and the labelling is a linear extension of
    $(\P_2,\preceq_-)$. 

    We focus on $\P_1$; the argument for $\P_2$ is analogous. Pick the maximal
    chain $D$ in $(\P_1,\preceq_+)$ starting in $a^+_k$. Then $A_{+i} \cap D
    \neq \emptyset$ for all $i > 0$ and hence $\{(g,t) \in \R^{\P_1}
    :\inner{\1_D,g} = 1\}$ is the hyperplane for the maximal simplex in the
    triangulation of $\Chain{\P_1,\preceq_+}$ corresponding to
    $A_{+0},\dots,A_{+k}$ and not containing the origin. Thus, one of the two
    hyperplanes supporting a facet of the simplex in $\TChain{\dP}$
    corresponding to $C$ is given by $H := \{ (g,t) \in \R^\P \times \R :
    \ell(g,t) = 1 \}$ where $\ell(g,t) = \inner{\1_D,g} - t$.
    
    Now, $\TPsi{\dP}$ is linear on the simplex $C$ in $\TChain{\dP}$ and can
    be easily inverted. Since $\dP$ is compatible, we can find a linear
    extension $\sigma : \P \rightarrow \{1,\dots,|\P|\}$ that respects the
    constructed linear extensions on $\P_1$ and $\P_2$. On the image of $C$
    under $\TPsi{\dP}$, the inverse is given by the linear transformation $T :
    \R^\P \times \R \to \R^\P \times \R$ with $T(f,t) = (f',t)$ and $f' : \P
    \rightarrow \R$ is defined as follows. If $b \in \P_1$, then by $f'(b) =
    f(b) - f(\bar b)$, $\bar b \prec_+ b$ is a cover relation and $\sigma(\bar
    b)$ is maximal. If $b \in \P_2$, we choose $\bar b$ covered by $b$ in with
    respect to $\preceq_-$. It can now be checked that $\ell \circ T = L_C$
    for some alternating chain $C$. Thus $H$ is supporting for $\TOrd{\dP}$
    and the map $\TPsi{\dP}$ maps $\TChain{\dP}$ onto $\TOrd{\dP}$.
\end{proof}

Theorem~\ref{thm:T_iso} does not extend to the non-compatible case as the
following example shows.

\begin{example}
    Consider the double poset $\dP = ([2],\le,\ge)$, that is, $\P_+$ is the
    $2$-chain $\{1,2\}$ and $\P_-$ is the opposite poset.  Then $\Chain{\P_+}
    = \Chain{\P_-} = T := \{ x \in \R^2 : x \ge 0, x_1 + x_2 \le 1\}$ and
    $\TChain{\dP}$ is a three-dimensional octahedron with volume
    $\frac{16}{3}$.  Any triangulation of the octahedron has at least four
    simplices.
    In contrast, $\Ord{\P_-} = \1 - \Ord{\P_+}$ and hence $\TOrd{\dP}$ is
    linearly isomorphic to a prism over a triangle with volume $4$. Any
    triangulation of the prism has exactly $3$ tetrahedra. 
\end{example}

\subsection{Volumes and Ehrhart polynomials}\label{ssec:TO_ehrhart}
The canonical subdivision of $\Ord{\P}$ makes it easy to compute its volume.
For a generic $f \in \Ord{\P}$, there is a unique linear extension $\sigma : \P
\rightarrow \{1,2,\dots,d\}$ where $d := |\P|$  such that 
\[
    \Ord{\P_f} \ = \ \{ h \in \R^P : 0 \le h(\sigma^{-1}(1)) \le \cdots \le
    h(\sigma^{-1}(d)) \le 1 \}.
\]
In particular, the full-dimensional simplex $\Ord{\P_f}$ is unimodular
relative to $\Z^\P \subseteq \R^\P$ and has volume $\vol(\Ord{\P_f}) =
\frac{1}{|\P|!}$. If we denote by $e(\P)$ the number of linear extensions of
$\P$, then Stanley~\cite{TwoPoset} showed the following.

\begin{cor}\label{cor:O_vol}
    $\nvol(\Ord{\P}) = |\P|! \cdot \vol(\Ord{\P}) = e(\P)$.
\end{cor}

For the Ehrhart polynomial $\ehr_{\Ord{\P}}(n)$ of $\Ord{\P}$ it suffices to
interpret the lattice points in $n \Ord{\P}$ for $n > 0$. Every point in
$n\Ord{\P} \cap \Z^P$ corresponds to an order preserving map $\phi : \P
\rightarrow [n+1]$. Counting order preserving maps is
classical~\cite[Sect.~3.15]{EC1new}: the \Defn{order polynomial}
$\Omega_\P(n)$ of $\P$ counts the number of order preserving maps into
$n$-chains. The \Defn{strict} order polynomial $\Omega^\circ_\P(n)$ counts the
number of strictly order preserving maps $f : \P \rightarrow [n]$, that is,
$f(a) < f(b)$ for $a \prec b$.  The transfer map $\Transfer_\P$ as well as its
inverse $\iTransfer_\P$ (given in~\eqref{eqn:transfer}
and~\eqref{eqn:itransfer}, respectively) both take lattice points to lattice
points and hence, together with Theorem~\ref{thm:EM}, yield the following
result.

\begin{cor}\label{cor:O_C_iso}
    Let $\P$ be a finite poset. Then for every $n > 0$
    \begin{align*}
        \Omega_\P(n+1)
        & \ = \ 
        \ehr_{\Ord{\P}}(n) \ = \ \ehr_{\Chain{\P}}(n)\\
    \intertext{ and }
        (-1)^{|\P|}\Omega^\circ_\P(n-1) &\ = \ 
        \ehr_{\Ord{\P}}(-n) \ = \ \ehr_{\Chain{\P}}(-n).
    \end{align*}
    In particular, $\vol(\Ord{\P}) = \vol(\Chain{\P})$.
\end{cor}
This is an interesting result as it implies that the number of linear
extensions of a poset $\P$ only depends on the comparability graph $G(\P)$.

\begin{thm}\label{thm:TC_ehr}
    Let $\dP = (\P,\preceq_+,\preceq_-)$ be a double poset. Then
    $\DChain{\dP}$ is a lattice polytope with respect to $\Z^\P$ and
    \begin{align*}
    \ehr_{\DChain{\dP}}(n-1) \ &= \ \sum_{\P = \P_1 \uplus \P_2}
        \Omega^\circ_{(\P_1,\preceq_+)}(n-1) \cdot
        \Omega_{(\P_2,\preceq_-)}(n) \text{ and }\\
        \nvol(\DChain{\dP}) \ &= \ 
        \sum_{\P = \P_1 \uplus \P_2} \binom{|\P|}{|\P_1|} 
        e(\P_1,\preceq_+)\cdot
        e(\P_2,\preceq_-).
    \end{align*}
\end{thm}
\begin{proof}
    Since $\Chain{\P} = \Po_{G(\P)}$ is a dual integral anti-blocking
    polytope, the first identity follows from Corollary~\ref{cor:AB_ehrhart}
    and Corollary~\ref{cor:O_C_iso}. The second identity follows
    from Corollary~\ref{cor:AB_volumes} and Corollary~\ref{cor:O_vol}.
\end{proof}

Notice from Theorem~\ref{thm:Ehr_AB_Cay} we can also deduce a closed formula
for the Ehrhart polynomial of $\TChain{\dP}$ with respect to the lattice
$\Z^\P \times \Z$ and, by substituting $\frac{1}{2}k$ for $k$, also with
respect to the affine lattice $\AffLat$.  These formulas are not very
enlightening and instead we record the normalized volume. Note that the
minimal Euclidean volume of a full-dimensional simplex with vertices in
$\AffLat = \Z^\P \times (2\Z + 1)$ is $\frac{2^{|P|+1}}{(|P|+1)!}$.

\begin{cor}\label{cor:TC_Vol}
    Let $\dP = (\P,\preceq_+,\preceq_-)$ be a double poset. Then
    the normalized volume with respect to the affine lattice $\AffLat = 2\Z^\P
    \times (2\Z + 1)$ is 
    \[
        \nvol(\TChain{\dP}) \ = \ 
        \sum_{\P = \P_1 \uplus \P_2} 
        e(\P_1,\preceq_+) \cdot
        e(\P_2,\preceq_-).
    \]
\end{cor}

We leave it to the reader to give direct combinatorial interpretations of the
volume and the Ehrhart polynomials for double posets. 

It follows directly from~\eqref{eqn:TPsi} that $\TPsi{\P_{\pm}} : \R^\P
\rightarrow \R^\P$ maps lattice points to lattice points. If $\dP$ is
compatible, then the proof of Lemma~\ref{lem:TPsi_iso} asserts that
$\TPsi{\P_{\pm}}$ is in fact lattice preserving. Hence, we record an analog
to Corollary~\ref{cor:O_C_iso}.

\begin{cor}\label{cor:TO_TC_iso}
    If $\dP$ is a compatible double poset, then $\TOrd{\dP}$ and
    $\TChain{\dP}$ have the same Ehrhart polynomials and hence the same
    volumes.
\end{cor}

The formulas of Theorem~\ref{thm:TC_ehr} are particularly simple when $\dP$
is special or anti-special. We illustrate these cases at some simple examples.

\begin{example}
    For the 'XW'-double poset we have 
    \[
        \nvol(\TOrd{\dP_{XW}}) \ = \  \nvol(\TChain{\dP_{XW}}) 
        \ = \  \tfrac{6!}{2^6} \vol(\TChain{\dP_{XW}}) 
        \ = \ 128
    \]
    and 
    $\nvol(\DOrd{\dP_{XW}})= \nvol(\DChain{\dP_{XW}})
    = 6! \vol(\DChain{\dP_{XW}})= 880$.
\end{example}

\begin{example} 
    As the following examples are all compatible, the given values
    also give the normalized volumes of the respective (reduced) double order
    polytopes.
    \begin{enumerate}
        \item Let $\dP = ([d],\le,\le)$ be the double chain on $d$
            elements.  Then $\TChain{\dP}$ is a crosspolytope and 
            $\nvol(\TChain{\dP})  = 2^d$ and it follows
            from Vandermonde's identity that 
            \[
                \nvol(\DChain{\dP}) \ = \ d! \vol(\DChain{\dP}) \ = \
                \sum_{i=0}^d \binom{d}{i}^2 \ = \ \binom{2d}{d}.
            \]
        \item If $\dP$ is the double anti-chain on $d$ elements, then
            $\TChain{\dP}$ is isomorphic to $[0,2]^{d} \times [-1,1]$ and its
            normalized volume is
            \[
                \nvol(\TChain{\dP}) \ = \ 
                \tfrac{(d+1)!}{2^{d+1}}\vol(\TChain{\dP}) \ = \ 
                \sum_{i=0}^d \binom{d}{i} i!
                (d-i)! \ = \ (d+1)!.
            \]
            Likewise, $\DChain{\dP}$ is isomorphic to
            $[-1,1]^d$, which can be decomposed into $2^d$ unit cubes.
            Consequently, its normalized volume is 
            \[
                \nvol(\DChain{\dP})  \ = \ \sum_{i=0}^d \binom{d}{i}^2 i!(d-i)! \ = \ 2^d
            d!.
            \]
        \item Let $\dP$ be the double poset such that $\P_+$ is the
            $d$-chain and $\P_-$ is the $d$-antichain. Then
            \[
                \nvol(\TChain{\dP}) \ = \ \sum_{i=0}^d \frac{d!}{i!} 
            \]
            is the number of choices of ordered subsets of a $d$-set.
            Moreover
            \[
                \nvol(\DChain{\dP})  \ = \ \sum_{i=0}^d
                \binom{d}{i}^2 i! 
            \]
            is the number of partial permutation matrices, i.e.\
            $0/1$-matrices of size $d$ with at most one nonzero entry per row
            and column. Indeed, such a matrix is uniquely identified by an
            $i$-by-$i$ permutation matrix and a choice of $i$ rows and $i$
            columns in which it is embedded.
        \item For the comb $C_n$, the number of linear extensions is $e(C_n) =
            (2n-1)!!$. Let $\dP$ be the double poset induced by the comb
            $C_n$. Then an induction argument shows that
            \[
                \nvol(\TChain{\dP}) \ = \ 4^n\, n!.
            \]
            It would be nice to have a bijective proof of this equality.
    \end{enumerate}
\end{example}

Let $\dP_\circ = (\P,\preceq,\preceq)$ be a compatible double poset induced by
a poset $(\P,\preceq)$.  By Corollary~\ref{cor:TO_TC_iso}, the polytopes
$\TOrd{\dP_\circ}$ and $\TChain{\dP_\circ}$ have the same normalized volume.
Since both polytopes are $2$-level, this means that the number of maximal
simplices in any pulling triangulation of $\TOrd{\dP_\circ}$ and
$\TChain{\dP_\circ}$ coincides.  From Theorem~\ref{thm:twisted_val}, we know
that $\TOrd{\dP_\circ}^\dual$ is the twisted prism over the valuation polytope
associated to $\P$. On the other hand, we know from Corollary~\ref{cor:Hansen}
that $\TChain{\dP_\circ}^\dual$ is linearly isomorphic to the Hansen polytope
$\Hansen{\overline{G(\P)}}$. Moreover, $\TOrd{\dP_\circ}^\dual$ and
$\TChain{\dP_\circ}^\dual$ are both $2$-level and it is enticing to conjecture
that their normalized volumes also agree. Unfortunately, this is not the case.
For the poset $\P$ on $5$ elements whose Hasse diagram is the letter 'X', any
pulling triangulation of $\TChain{\dP_\circ}^\dual$ has $324$ simplices whereas
for $\TOrd{\dP_\circ}^\dual$ pulling triangulations have $320$ simplices.

\section{Gr\"obner bases and triangulations}\label{sec:GB}

\subsection{Double Hibi rings}\label{ssec:double_hibi}
Hibi~\cite{Hibi87} associated to any finite poset $(\P,\preceq)$ a ring
$\Hring{\P}$, nowadays called \Defn{Hibi ring}, that algebraically reflects
many of the order-theoretic properties of $\P$. The ring $\Hring{\P}$ is
defined as the graded subring of the polynomial ring $S = \C[t, s_a : a \in \P]$
generated by the elements $t \cdot s^\Filter$, where
\[
    s^\Filter \ := \  \prod_{a \in \Filter} s_a,
\]
ranges over all filters $\Filter \subseteq \P$. For example, Hibi showed
that $\Hring{\P}$ is a normal Cohen--Macaulay domain of dimension $|\P|+1$ and
that $\Hring{\P}$ is Gorenstein if and only if $\P$ is a graded poset. By definition, Hibi rings are toric and hence they have the following quotient description. Let $R =
\C[x_\Filter : \Filter \in \Birk{\P}]$ be the polynomial ring with variables
indexed by filters and define the homogeneous ring map $\phi : R \rightarrow
S$ by $\phi(x_\Filter) = t\, s_\Filter$. Then $\Hring{\P} \cong R /
\Id_{\Ord{\P}}$
where $\Id_{\Ord{\P}} = \ker \phi$ is a toric ideal. 

Hibi elegantly described a reduced Gr\"obner basis of $\Id_{\Ord{\P}}$ in terms of
$\Birk{\P}$.  Fix a total order $\le$ on the variables of $R$ such that
$x_\Filter \le x_{\Filter'}$ whenever $\Filter \subseteq \Filter'$ and let
$\mo$ denote the induced reverse lexicographic order on $R$. For $f \in R$,
we write $\init(f)$ for its leading term with respect to $\mo$ and we will
underline leading terms in what follows.

\begin{thm}[{\cite[Thm.~10.1.3]{HH11}}]\label{thm:HibiIdeal}
    Let $(\P,\preceq)$ be a finite poset.
    Then the collection
    \begin{equation}\label{eqn:Hibi_rels}
        \underline{x_{\Filter} \, x_{\Filter'}} \ - \
        x_{\Filter\cap\Filter'} \, x_{\Filter\cup\Filter'} \quad \text{
        with }\Filter,\Filter'\in\Birk{\P}\text{ incomparable}
    \end{equation}
    is a reduced Gr\"obner basis of $\Id_{\Ord{\P}}$.
\end{thm}
The binomials~\eqref{eqn:Hibi_rels} are called \Defn{Hibi relations}.

In light of the previous sections, the natural question that we will address
now is regarding an algebraic counterpart of the Hibi rings for double posets.
For a double poset $\dP = (\P,\preceq_+,\preceq_-)$, we define the
\Defn{double Hibi ring} $\dHring{\dP}$ as the subalgebra of the Laurent ring
$\hat{S} := \C[t_-,t_+,s_a,s_a^{-1} : a \in \P]$ spanned by the elements $t_+
\cdot s^{\Filter}$ for filters $\Filter \in \Birk{\P_+}$ and  $t_- \cdot
(s^\Filter)^{-1}$ for filters $\Filter \in \Birk{\P_-}$. This is the affine
semigroup ring associated to $\TOrd{\dP}$ with respect to the affine lattice
$\AffLat = 2\Z^\P \times (2\Z + 1)$. Up to a translation by $(\0,1)$ and the
lattice isomorphism $2\Z^\P \times 2\Z \cong \Z^\P \times \Z$, the double Hibi ring $\dHring{\dP}$
is the affine semigroup ring of
\[
    \conv \bigl\{ (\Ord{\P_+} \times \{1\} ) \cup (-\Ord{\P_-} \times \{0\} )
    \bigr\},
\]
with respect to the usual lattice $\Z^\P \times \Z$.  In particular, the
double Hibi ring $\dHring{\dP}$ is graded of Krull dimension $|\P|+1$.
Moreover, since the double order polytope $\TOrd{\dP}$ is reflexive by
Corollary \ref{cor:TO_ineq}, it follows that $\dHring{\dP}$ is a Gorenstein
domain for any compatible double poset $\dP$. The rings $\dHring{\dP}$ as well
as affine semigroup rings associated to the double chain polytopes
$\TChain{\dP}$ as treated at the end of Section~\ref{ssec:GB_triang} were also
considered by Hibi and Tsuchiya~\cite{HT}. 

Set $\hat{R} := \C[ x_{\Filter_+},
x_{\Filter_-} : \Filter_+ \in \Birk{\P_+}, \Filter_+ \in \Birk{\P_+}]$ and
define the monomial map $\hat \phi : \hat R \rightarrow \hat S$ by
\[
    \hat \phi(x_{\Filter_+}) \ = \  t_+\, s^{\Filter_+}
    \quad \text{ and } \quad
    \hat \phi(x_{\Filter_-}) \ = \  t_-\, (s^{\Filter_-})^{-1}.
\]
The corresponding toric ideal $\Id_{\TOrd{\dP}} = \ker \hat \phi$ is then
generated by the binomials
\begin{equation}\label{eqn:binom}
    \underline{
    x_{\Filter_{+1}} x_{\Filter_{+2}} \dots x_{\Filter_{+k_+}} \cdot
    x_{\Filter_{-1}} x_{\Filter_{-2}} \dots x_{\Filter_{-k_-}}}
    -
    x_{\Filter'_{+1}} x_{\Filter'_{+2}} \dots x_{\Filter'_{+k_+}} \cdot
    x_{\Filter'_{-1}} x_{\Filter'_{-2}} \dots x_{\Filter'_{-k_-}},
\end{equation}
for filters $\Filter_{+1}, \dots, \Filter_{+k_+}, \Filter'_{+1}, \dots,
\Filter'_{+k_+} \in \Birk{\P_+}$ and $\Filter_{-1}, \dots, \Filter_{-k_-},
\Filter'_{-1}, \dots, \Filter'_{-k_-} \in \Birk{\P_-}$.

Again, fix a total order $\le$ on the variables of $\hat{R}$ such that for
$\sigma = \pm$
\begin{compactitem}
    \item $x_{\Filter_\sigma} < x_{\Filter'_\sigma}$ for any filters
        $\Filter_\sigma,\Filter'_\sigma \in \Birk{\P_\sigma}$ with
        $\Filter_\sigma \subset \Filter'_\sigma$, and
    \item $x_{\Filter_+} < x_{\Filter_-}$ for any filters 
        $\Filter_+ \in \Birk{\P_+}$ and 
        $\Filter_- \in \Birk{\P_-}$,
\end{compactitem}
and denote by $\mo$ the reverse lexicographic term order on $\hat R$ induced
by this order on the variables.

\begin{thm}\label{thm:TO_GB}
    Let $\dP = (\P,\preceq_+,\preceq_-)$ be a compatible double poset. Then
    a Gr\"obner basis for $\Id_{\TOrd{\dP}}$ is given by the binomials
    \begin{align}
        \label{eqn:dHibi_rels}
        \underline{x_{\Filter_\sigma}\, x_{\Filter_\sigma'}} \ &- \
        x_{\Filter_\sigma \cup \Filter_\sigma'} \, x_{\Filter_\sigma \cap
        \Filter_\sigma'}\\
        \intertext{for incomparable filters $\Filter_\sigma,\Filter_\sigma'
        \in \Birk{\P_\sigma}$ and $\sigma = \pm$, and }
        \label{eqn:twist_rels}
        \underline{x_{\Filter_+}\, x_{\Filter_-}} \ &- \ x_{\Filter_+\setminus
        A} \, x_{\Filter_-\setminus A}
    \end{align}
    for filters $\Filter_+ \in \Birk{\P_+}, \Filter_- \in \Birk{\P_-}$ such
    that $A := \min(\Filter_+)\cap\min(\Filter_-) \neq \emptyset$.
\end{thm}

It is clear that binomials of the form~\eqref{eqn:dHibi_rels}
and~\eqref{eqn:twist_rels} are contained in $\Id_{\TOrd{\dP}}$ and hence it
suffices to show that their leading terms generate $\init(\Id_{\TOrd{\dP}})$.  For
this, let us take a closer look at the combinatorics of $\hat \phi$. Let
$\mathcal G$ be the collection of binomial given in~\eqref{eqn:dHibi_rels}
and~\eqref{eqn:twist_rels} and let $f = \underline{m_1} - m_2$ be an element
of the form~\eqref{eqn:binom}. By reducing $f$ by the binomial
of~\eqref{eqn:dHibi_rels}, we can view $f$ as a quadruple
\begin{equation}\label{eqn:quad}
\begin{aligned}
    \Filter_{+1} \subset \Filter_{+2} \subset \cdots \subset \Filter_{+k_+} &\quad\qquad
    \Filter_{-1} \subset \Filter_{+2} \subset \cdots \subset \Filter_{-k_-}\\
    \Filter_{+1}' \subset \Filter_{+2}' \subset \cdots \subset \Filter_{+k_+}' &\quad\qquad
    \Filter_{-1}' \subset \Filter_{+2}' \subset \cdots \subset
    \Filter_{-k_-}'.
\end{aligned}
\end{equation}
From the definition of $\hat \phi$ it follows that such a quadruple defines a
binomial in $\Id_{\TOrd{\dP}}$ if and only if for any $q \in \P$
\begin{equation}\label{eqn:moving}
        \max\{r:q \notin \Filter_{+r} \} - \max\{s : q \notin \Filter_{-s} \}
        \ = \ \max\{ r : q \notin \Filter_{+r}' \} - \max\{ s : q \notin
        \Filter_{-s}'\}.
\end{equation}
and we note the following implication.
\begin{lem}\label{lem:moving}
    Let the collection of filters in~\eqref{eqn:binom} correspond to a
    binomial $f \in \Id_{\TOrd{\dP}}$ and let $q \in \P$. Then there is some $1
    \le i \le k_+$ such that $q \in \Filter_{+i} \setminus \Filter_{+i}'$ if
    and only if there is some $1 \le j \le k_-$ such that 
    $q \in \Filter_{-j} \setminus \Filter_{-j}'$.
\end{lem}
\begin{proof}
    If $q \in \Filter_{+i} \setminus \Filter_{+i}'$, then $\max\{r:q \notin
    \Filter_{+r} \} < i$ and $\max\{r:q \notin \Filter_{+r}' \} \ge i$
    and~\eqref{eqn:moving} implies that $q \in \Filter_{-j} \setminus
    \Filter_{-j}'$ for some $j$. The other direction is identical.
\end{proof}

We call $q \in \P$ \Defn{moving} if it satisfies one of the two equivalent
conditions of Lemma~\ref{lem:moving}.

\begin{proof}[Proof of Theorem~\ref{thm:TO_GB}]
    Let $f = \underline{m_1} - m_2 \in \Id_{\TOrd{\dP}}$ be a binomial represented
    by a collection of filters given by~\eqref{eqn:quad}. If $k_- = 0$ or $k_+
    = 0$, then the Hibi relations~\eqref{eqn:dHibi_rels} for $\P_-$ or $\P_+$
    together with Theorem~\ref{thm:HibiIdeal} yields the result. Thus, we
    assume that $k_-,k_+ > 0$ and we need to show that there are filters
    $\Filter_{+i}$ and $\Filter_{-j}$ such that $\min(\Filter_{+i}) \cap
    \min(\Filter_{-j}) \neq \emptyset$.

    Observe that there is at least one moving element. Indeed,
    $\Filter_{+1}\nsubseteq \Filter_{+1}'$ and hence $\Filter_{+1} \setminus
    \Filter_{+1}' \neq \emptyset$. Otherwise, $x_{\Filter_{+1}} <
    x_{\Filter_{+1}'}$ and the reverse lexicographic term order $\mo$ would
    not select $m_1$ as the lead term of $f$.  Among all moving elements,
    choose $q$ to be minimal with respect to $\preceq_+$ and $\preceq_-$.
    Since $\dP$ is a compatible double poset, such a $q$ exists. But then,
    if $q \in \Filter_{+i} \setminus \Filter_{+i}'$, then $q \in
    \min(\Filter_{+i})$. The same holds true for $\Filter_{-j}$ and shows that
    $\underline{m_1}$ is divisible by the leading term of a binomial of
    type~\eqref{eqn:twist_rels}.
\end{proof}

\newcommand\mkFilter[1]{\langle #1 \rangle}

\subsection{Gr\"obner bases, faces, and triangulations}\label{ssec:GB_triang}
In light of the regular and unimodular triangulation of $\Ord{\P}$ given
in~\cite{TwoPoset} (and recalled in Section~\ref{ssec:TO_triang}), the Hibi
ring $\Hring{\P}$ is exactly the affine semigroup ring associated to
$\Ord{\P}$. That is, $\Hring{\P}$ is the standard graded $\C$-algebra associated to the
normal affine semigroup 
\[
    \{ (f,k) \in \Z^\P \times \Z : k \ge 0, f \in k \Ord{\P}\}.
\]

For a lattice polytope $\Po \subset \R^n$,
Sturmfels~\cite[Thm.~8.3]{Sturmfels96} described a beautiful relationship
between regular triangulations of $\Po$ and radicals of initial ideals of the
toric ideal $\Id_\Po$. It follows from Theorem~\ref{thm:TO_GB} that
$\init(\Id_{\TOrd{\dP}})$ is a squarefree ideal generated by quadratic monomials.
Appealing to~\cite[Thm.~8.3]{Sturmfels96}, this yields the following
refinement of Theorem~\ref{thm:T_iso}.

\begin{cor}\label{cor:TO_triang}
    Let $\dP$ be a compatible double poset. Then $\TOrd{\dP}$ has a
    regular triangulation whose underlying simplicial complex is exactly
    $\niDelta{\dP}$.
\end{cor}
\begin{proof}
    The initial ideal $\init(\Id_{\TOrd{\dP}})$ is already radical and Theorem 8.3
    of~\cite{Sturmfels96} yields that $\init(\Id_{\TOrd{\dP}})$ is the
    Stanley-Reisner ideal of a regular triangulation of $\TOrd{\dP}$.
    Hence, a collection $C = C_1 \uplus C_2 \subseteq \TBirk{\dP}$ forms a
    simplex in the triangulation of $\TOrd{\dP}$ if and only if 
    \[
        \prod_{\Filter_+ \in C_+} x_{\Filter_+}
        \prod_{\Filter_- \in C_-} x_{\Filter_-} \ \not\in \init(\Id_{\TOrd{\dP}}).
    \]
    Translating the conditions given in Theorem~\ref{thm:TO_GB}, this is the case
    if and only if $C_\sigma = C \cap \TBirk{\P_\sigma}$ is a chain of filters
    for $\sigma = \pm$ and $C_+, C_-$ are non-interfering chains. This is
    exactly the definition of the flag complex $\niDelta{\dP}$.
\end{proof}

Using the orbit-cone correspondence for affine toric varieties (see, for
example, \cite[Sect.~3.2]{CLS}), we can give an algebraic perspective on
Theorem~\ref{thm:TO_faces}. We are in a particularly nice situation as the
polytopes we consider have unimodular triangulations and hence the affine
semigroup rings are generated in degree $1$ by the vertices of the underlying
polytope.

\begin{lem}\label{lem:cone_orbit}
    Let $V \subset \Lambda$ be a finite set of lattice points and $\Po =
    \conv(V)$ the corresponding lattice polytope. If $\Id \subseteq \C[ x_v : v
    \in V]$ is the toric ideal of the homogenization $\{ (v,1) : v
    \in V \} \subseteq \Lambda \times \Z$, then for any subset $U \subseteq V$, we have that $\conv(U)$ is
    a face of $\Po$ with $\conv(U)\cap V = U$ if and only if
    \[
        f(\1_U) \ = \ 0 \qquad \text{ for all } f \in \Id.
    \]
\end{lem}

\begin{proof}[Proof of Theorem~\ref{thm:TO_faces}]
    Let $\L \subseteq \TBirk{\dP}$. Then for $\sigma = \pm$ and
    $\Filter_\sigma, \Filter_\sigma' \in \Birk{\P_\sigma}$
    Lemma~\ref{lem:cone_orbit} and~\eqref{eqn:dHibi_rels} of
    Theorem~\ref{thm:TO_GB} states that 
    \[
        \Filter_\sigma, \Filter_\sigma' \in \L_\sigma 
        \quad \Longleftrightarrow \quad
        \Filter_\sigma \cup \Filter_\sigma', 
        \Filter_\sigma \cap  \Filter_\sigma' \in \L_\sigma.
    \]
    That is, if and only if $\L_\sigma$ is an embedded Thus, $\L_\sigma$ is an
    embedded sublattice of $\Birk{\P_\sigma}$.  The same reasoning shows that
    the conditions imposed by~\eqref{eqn:twist_rels} are equivalent to those
    of Lemma~\ref{lem:cond2}.
\end{proof}

\newcommand\IChain[1]{\Id_{\TChain{#1}}}%
\newcommand\Cring[1]{\C[\TChain{#1}]}%
We can also use Sturmfels' result in the other direction to find Gr\"obner
bases. For a double poset $\dP = (\P,\preceq_+,\preceq_-)$ we may define
the
subring $\Cring{\dP} \subseteq \hat R$ generated by the monomials
$t_+ s^{\min(\Filter_+)}$  and $t_+ (s^{\min(\Filter_-)})^{-1}$ for filters
$\Filter_+ \subseteq \P_+$ and $\Filter_- \subseteq \P_-$. The corresponding
toric ideal $\IChain{\dP}$ is contained in the ring $T = \C[ x_{A_+},
x_{A_-} ]$, where $A_\sigma$ ranges over all anti-chains in $\P_\sigma$ for
$\sigma = \pm$.  Since $\TOrd{\dP}$ is the stable set polytope of the
perfect double graph $G(\dP)$, it follows from
Corollary~\ref{cor:TC_triang} that $\Cring{\dP}$ is the normal affine
semigroup ring associated to the lattice polytope $\TChain{\dP}$. To
describe a Gr\"obner basis for, we introduce the following notation. For
$\sigma = \pm$ and two antichains $A,A' \subseteq \P_\sigma$ define $A \sqcup A'
\ := \ \min(A \cup A')$ and 
\[
    A \sqcap A' \ := \ (A \cap A') \cup (\max(A \cup A') \setminus \min(A \cup
    A')).
\]
For a subset $S \subseteq
\P$ and $\sigma = \pm$, we write $\mkFilter{S}_\sigma := \{ a \in \P : a
\succeq_\sigma s \text{ for some } s \in S \}$ for the filter in $\P_\sigma$
generated by $S$.

\begin{thm}\label{thm:TC_GB}
    Let $\dP$ be a double poset. Then a Gr\"obner basis for
    $\IChain{\dP}$ is given by the binomials
    \begin{align*}
        \underline{x_A \, x_{A'}} \ &- \ 
        x_{A \sqcup A'}\, x_{A \sqcap A'} \qquad \mkFilter{A}_\sigma,
        \mkFilter{A}_\sigma \in \Birk{\P_\sigma} \text{ incomparable}\\
        \intertext{ for antichains $A,A' \subset \P_\sigma$ for $\sigma = \pm$
        and }
        \underline{x_{A_+} \, x_{A_-}} \ &- \ x_{A_+ \setminus
        A_-}x_{A_-\setminus A_+}
        \qquad \text{ for antichains } A_\sigma \subseteq \P_\sigma.
    \end{align*}
\end{thm}
\begin{proof}
    It is easy to verify that the given binomials are contained in
    $\IChain{\dP}$. Moreover, the leading monomials are exactly the minimal
    non-faces of the unimodular triangulation of Corollary~\ref{cor:TC_triang}.
    The result now follows from Theorem 8.3 in~\cite{Sturmfels96}.
\end{proof}

\begin{rem}
    Reformulated in the language of double posets, Hibi, Matsuda, and
    Tsuchiya~\cite{Hibi15-1,Hibi15-2,Hibi15-3} computed related Gr\"obner
    bases of the toric ideals associated with the polytopes
    $\Gamma(\Ord{\P_+},\Ord{\P_-})$ (in the compatible case),
    $\Gamma(\Chain{\P_+},\Chain{\P_-})$, and $\Gamma(\Ord{\P_+},\Chain{\P_-})$
    for a double poset $\dP$. See the paragraph before
    Corollary~\ref{cor:Gamma} for notation.
\end{rem}

\bibliographystyle{siam}
\bibliography{TwistedPosetPolytopes}

\end{document}